\documentclass[12pt, reqno]{amsart}
\usepackage[all,tips]{xy}
\usepackage{latexsym, amsfonts, amsmath, amssymb, mathrsfs, graphicx, xcolor, ytableau, hyperref, float, xurl}
\allowdisplaybreaks

\definecolor{red}{rgb}{1,0,0}
\definecolor{magenta}{rgb}{1,0,1}
\definecolor{dartmouthgreen}{rgb}{0.05, 0.5, 0.06}
\definecolor{purple(x11)}{rgb}{0.63,0.36,0.94}
\definecolor{turquoise}{rgb}{0.25, 0.87, 0.81}
\newtheorem{theorem}{Theorem}[section]
\newtheorem{lemma}[theorem]{Lemma}
\newtheorem{proposition}[theorem]{Proposition}
\newtheorem{corollary}[theorem]{Corollary}
\newtheorem{conjecture}[theorem]{Conjecture}

\theoremstyle{definition}
\newtheorem{remark}[theorem]{Remark}
\newtheorem{example}[theorem]{Example}

\newcommand{\Cal}[1]{\ensuremath{\mathcal{#1}}}

\newcommand{\lp}{\left(}
\newcommand{\rp}{\right)}

\def\C{{\mathbb C}}
\def\R{{\mathbb R}}
\def\N{{\mathbb N}}
\def\Z{{\mathbb Z}}
\def\Q{{\mathbb Q}}

\def\O_K{{\Cal{O}_{K}}}
\def\O_F{{\Cal{O}_{F}}}
\def\N_F{{\Cal{N}_{F/\Q}}}

\def\CC{{\mathcal C}}

\def\O_K{{\Cal{O}_{K}}}
\def\O_F{{\Cal{O}_{F}}}
\def\N_F{{\Cal{N}_{F/\Q}}}

\numberwithin{equation}{section}
\numberwithin{theorem}{section}

 \title{Hook length biases and general linear partition inequalities}
 \author{Cristina Ballantine}
 \address{Department of Mathematics and Computer Science, College of the Holy Cross, Worcester, MA 01610, USA}
  \email{cballant@holycross.edu}
 \author{Hannah Burson}
 \address{School of Mathematics, University of Minnesota, Twin Cities, 127 Vincent Hall 206 Church St. SE, 
Minneapolis, MN 55455, USA}
\email{hburson@umn.edu}
\author{William Craig}
\address{Mathematical Institute,
University of Cologne,
Gyrhofstr. 8b,
50931 Cologne, Germany}
\email{wcraig@uni-koeln.de}
 \author{Amanda Folsom}
 \address{Department of Mathematics and Statistics, Amherst College, Amherst, MA 01002, USA}
 \email{afolsom@amherst.edu}
 \author{Boya Wen}
 \address{Department of Mathematics, University of Wisconsin, 480 Lincoln Drive, Madison, WI 53706, USA}
\email{bwen25@wisc.edu}
\thanks{Acknowledgements.  Craig was partially supported by the Thomas Jefferson Fund at the University of Virginia and the NSF (DMS-1601306 and DMS-2055118) and received funding from the European Research Council (ERC)
under the European Union’s Horizon 2020 research and innovation programme (grant agreement
No. 101001179).  Folsom was partially supported by NSF Grant DMS-2200728. Wen was partially supported by NSF Grant DMS-1928930 while she was in residence at the Simons Laufer Mathematical Sciences Institute (formerly MSRI) in Berkeley, California, during the Spring 2023 semester. The authors also acknowledge support from the Shanahan Fund at the College of the Holy Cross. }

 
\begin{document} 
 \maketitle

\begin{abstract} Motivated in part by hook-content formulas for certain restricted partitions in representation theory, we consider the total number of hooks of fixed length in  odd versus  distinct partitions. We show that there are more hooks of length $2$, respectively $3$, in all odd partitions of $n$ than in all distinct partitions of $n$,  and make the analogous conjecture for arbitrary hook length $t \geq 2$. We also establish additional bias results on the number of gaps of size $1,$ respectively $2$, in all odd versus distinct partitions of $n$. We conjecture similar biases and asymptotics, as well as congruences for the number of hooks of fixed length in odd distinct partitions versus self-conjugate partitions. 

An integral component of the proof of our bias result for hooks of length $3$ is a linear inequality involving $q(n)$, the number of distinct partitions of $n$.  In this article  we also establish effective linear inequalities  for $q(n)$ in great generality, a result  which is  of independent interest.  

Our methods are both analytic and combinatorial, and our results and conjectures intersect the areas of representation theory, analytic number theory,  partition theory, and $q$-series. In particular, we use a Rademacher-type exact formula for $q(n),$   Wright's circle method, modularity,  $q$-series transformations, asymptotic methods, and combinatorial arguments.
\end{abstract}

 \section{Introduction} \label{intro}
  Connections between representation theory and the theory of integer partitions are well-known.  For example, 
the irreducible  polynomial representations of $\operatorname{GL}_n(\mathbb C)$ may be indexed by partitions of length at most $n$;  moreover, the conjugacy classes of the symmetric group $S_n$, and therefore the number of non-equivalent irreducible complex representations, may be indexed by the partitions of $n$.  \emph{Hook lengths} of partitions play particularly important roles in establishing these connections.  Namely, such irreducible representations   can be analyzed via partition \emph{Young tableaux}.  The dimension of  a representation of $S_n$ (respectively $GL_n(\mathbb C)$)  corresponding to a particular partition is given by a   \emph{hook length formula} (respectively a \emph{hook-content formula}).  
For more on these topics, see e.g. \cite{stanley}.  

 We recall that a \emph{partition}   $\lambda=(\lambda_1, \lambda_2, \ldots, \lambda_\ell)$ of \emph{size} $n\in\mathbb N_0$  is a non-increasing sequence of positive integers $\lambda_1\geq \lambda_2 \geq \cdots \geq \lambda_\ell$ called \emph{parts} that add up to $n$. The number of parts of a partition $\lambda$ is called the \emph{length} of $\lambda$ and is denoted by $\ell(\lambda)$. We denote by $\mathcal P$ the set of all partitions, and by $p(n)$ the number of partitions of $n$. We adopt the usual convention that the empty set is the only partition of zero.

A partition $\lambda=(\lambda_1, \lambda_2, \ldots, \lambda_\ell)$ has a natural graphical representation as a \emph{Young diagram} (also called a \emph{Ferrers diagram}), i.e. a left-justified vertical array of boxes with $\lambda_i$ boxes in the $i$-th row from the top. The \emph{conjugate} of a partition $\lambda$ is the partition $\lambda'$ whose Young diagram has the columns of $\lambda$ as rows. 
   Each box in a Young diagram of $\lambda$ may be labeled with a \emph{hook number}, also called \emph{hook length}, which, informally, is the number of boxes in the upside-down-L-shaped portion of  the diagram with  the box appearing as its corner.  More  precisely,  
   for a box   in the $i$-th row and $j$-th column of the Young diagram of a partition $\lambda$, its  hook length  is defined as $h(i,j)=\lambda_i+\lambda'_j-i-j+1$ (see Figure \ref{fig1}).   We denote by $\mathcal H(\lambda)$ the multi-set of all hook lengths of $\lambda$, and let $$\mathcal H_t(\lambda) := \{h \ | \ h \in \mathcal H(\lambda), \ h\equiv 0 \pmod{t}\}.$$ 
   If   $\mathcal H_t(\lambda) = \emptyset,$ then $\lambda$ is called a \emph{$t$-core partition}.

In their study of Seiberg-Witten theory, Nekrasov and Okounkov \cite{NO} discovered the now 
  celebrated   formula for arbitrary powers of Euler's infinite product in terms of hook numbers.
\begin{theorem}[Nekrasov-Okounkov] For any complex number $z$ we have
\begin{align}\label{eqn_NO} \sum_{\lambda \in \mathcal P} x^{|\lambda|}\prod_{h \in \mathcal H(\lambda)} \left(1-\frac{z}{h^2}\right) = \prod_{k=1}^\infty (1-x^k)^{z-1}.\end{align}
\end{theorem}

Using properties of a classical combinatorial bijection between partitions and $t$-cores and $t$-quotients, Han  \cite{han}   established an  extension of the Nekrasov-Okounkov formula \eqref{eqn_NO},  which unifies the Macdonald identities in representation theory and the generating function for  {$t$-core partitions}. 
\begin{theorem}[Han] Let $t$ be a positive integer.  For any complex numbers $y$ and $z$ we have
\begin{align}\label{eqn_Han2} \sum_{\lambda \in \mathcal P} x^{|\lambda|} \prod_{h \in \mathcal H_t(\lambda)} \left(y-\frac{tyz}{h^2}\right) = \prod_{k=1}^\infty \frac{(1-x^{tk})^t}{(1-(yx^t)^k)^{t-z}(1-x^k)}.\end{align}
\end{theorem}
   
In a similar manner, Han \cite{han} also obtained a two-variable generating function for the number of partitions of $n$ with $m$ hooks of length $t$.

\begin{theorem}[Han] Let $t$ be a positive integer. For any complex number $y$ we have \begin{align} \label{eqn_Han1} \sum_{\lambda \in \mathcal P} x^{|\lambda|}y^{\#\{h\in \mathcal H(\lambda), \ h=t\}} = \prod_{k=1}^\infty \frac{(1+(y-1)x^{tk})^t}{1-x^k}.\end{align}
\end{theorem}

\emph{Restricted} partitions also play important roles in this context.  For example, there is a relationship between irreducible spin representations of the symmetric group and  distinct partitions, with tableaux and hooks acting as  liaisons (see e.g. \cite{matsumoto, schur} for more). Han and Xiong have also established hook-content formulas for distinct partitions in  \cite{hx1, hx2}. Here, we compare certain hook numbers of distinct partitions to those of odd partitions.

  In what follows, we refer to a partition into odd parts as an odd partition and to a partition into distinct parts as a distinct partition.    We denote by $\mathcal O(n)$, respectively $\mathcal D(n)$, the set of odd, respectively distinct, partitions of $n$. Euler's identity \cite[Corollary 1.2]{And} states that $|\mathcal O(n)|=|\mathcal D(n)|$ for all $n\geq 0$. In this paper
  we establish results analogous to Han's generating function \eqref{eqn_Han2} for partitions in $\mathcal O(n),$ respectively $\mathcal D(n),$   for $t=2$ in \eqref{eqn_F2zq} and \eqref{eqn_G2zq}.  We do the same for $t=3$ in Section \ref{sec_t=3}, however our expressions for these generating functions are more complicated and are not manifestly positive.  We leave it as an open problem to find simpler, manifestly positive generating functions for $t=3$, and more generally for $t\geq 3$.   
  We use these results to study the total number of hooks of fixed length in all partitions in $\mathcal O(n)$, respectively $\mathcal D(n)$, as explained below.
    
 Let $a_t(n)$ (respectively\ $b_t(n)$) be the total number of hooks of length $t$ in all odd  (respectively\ distinct) partitions of $n$. 
 \begin{example} We compute $b_2(7)=6$ and $b_3(7)=6$. The partitions of $7$ into distinct parts are $(7), (6,1), (5,2), (4,3)$, and $(4,2,1)$.  We give their  Young diagrams (in that order) with hook lengths labeled.
\begin{figure}[H] \centering $${\small \begin{ytableau} 7&6&5&4&3&2&1 & \none&\none& 7&5&4&3&2&1 & \none&\none& 6&5&3&2&1 \\ \none&\none&\none&\none&\none&\none&\none&\none&\none & 1 &\none&\none&\none&\none&\none&\none&\none& 2&1 \\ \none \end{ytableau}}$$
$${\small \begin{ytableau} 5&4&3&1 &\none&\none&\none&\none& 6&4&2&1 \\ 3&2&1 &\none&\none&\none&\none&\none& 3&1 \\ \none&\none&\none&\none&\none&\none&\none&\none& 1 \end{ytableau}}$$
\caption{The distinct partitions of $n=7$ and their hook lengths.}
\label{fig1}
\end{figure} 
 \end{example}

For a partition $\lambda$, a box in its Young diagram  has hook length $1$ if and only if it is at the end of a row and there is no box directly below it. Thus, in a partition $\lambda$, the number of hooks of length $1$ equals the number of different part sizes in $\lambda$. 

The next result was conjectured by Beck \cite{OEISBeckConj} and proved analytically by Andrews \cite{AndrewsBeck}. 
\begin{theorem}[Andrews] \label{AB} The difference between the  total number of parts in all distinct partitions of $n$ and the total number of different part sizes in all odd partitions of $n$ equals $c(n)$, the number of partitions of $n$ with exactly one part occurring three times while all other parts occur only once. 
\end{theorem}

\begin{corollary}\label{cor_andt1} For $n\geq 0$, $b_1(n)-a_1(n)=c(n).$
\end{corollary}

Thus, there are at least as many hooks of length $1$ in all distinct partitions of $n$ as there are in all odd partitions of $n$.  Given Corollary \ref{cor_andt1} as well as the identity
\begin{align} \label{Conjecture Eqn}
    \sum_{t \geq 1} a_t(n) = \sum_{t \geq 1} b_t(n)
\end{align}
which follows from Euler's identity, it is  natural  to study the relationship between $a_t(n)$ and $b_t(n)$ for any fixed $t\geq 1$. Corollary \ref{cor_andt1} shows that $b_1(n)\geq a_1(n)$.  We  conjecture  that this bias  reverses for $t\geq 2$; that is, for large enough $n$ we conjecture that there are at least as many hooks of length $t$ in all odd partitions of $n$  as there are in all  distinct partitions of $n$. We state this conjecture formally below.

\begin{conjecture} \label{Odd-Distinct Conjecture}   \
\begin{itemize}
  \item[(i)] For every integer $t\geq 2,$ there exists an integer $N_t$ such that for all $n > N_t$, we have $a_t(n) \geq b_t(n)$. 
  Moreover,   we  conjecture the following values of $N_t$ for $2\leq t \leq 10$: 
\begin{figure}[H] \centering
\begin{tabular}{|c|c|c|c|c|c|c|c|c|c|} \hline
$t$ & 2 & 3 & 4 & 5 & 6 & 7 & 8 & 9 & 10 \\ \hline
$N_t$ & 0 & 7 & 8 & 18 & 16 & 34 & 34 & 56 & 59 \\ \hline
\end{tabular}
\caption{Conjectural values for $N_t$.}
\label{fig2}
\end{figure}
\item[(ii)] For every integer $t\geq 2$ we have that $a_t(n) - b_t(n) \to \infty$ as $n \to \infty$. 
\end{itemize}

\end{conjecture}
We note that for $n\leq N_t$ ($2\leq t \leq 10$),  some values of $a_t(n)-b_t(n)$ are   negative and some are nonnegative.  Data supporting Conjecture \ref{Odd-Distinct Conjecture} was {obtained} by enumerating partitions and not from generating functions; this is because the generating functions for $a_t(n)$ and $b_t(n)$ are  difficult to derive explicitly. 
We are, however, able to write down generating functions for $t=2$ and $t=3$, which we use to ultimately prove Conjecture \ref{Odd-Distinct Conjecture} for $t=2$ and $t=3$.

\begin{theorem}\label{diff}
We have $a_2(n) \geq b_2(n)$ for all $n \geq 0$ and $a_3(n) \geq b_3(n)$ for all $n > 7$.
\end{theorem}
\begin{theorem}\label{limit}
For $t \in \{2,3\}$, we have $a_t(n) - b_t(n) \to \infty$ as $n \to \infty$. 
\end{theorem}
\begin{remark}
Together, Theorems \ref{diff} and \ref{limit} prove that Conjecture \ref{Odd-Distinct Conjecture} is true when $t=2$ and $t=3$.
\end{remark} 

In order to prove Theorem \ref{diff} for $t=3$, we use a special case of our Theorem \ref{q(n) Katriel} which gives a   general result on linear inequalities for $q(n)$, the number of partitions of $n$ into distinct parts. The result of Theorem \ref{q(n) Katriel} is also of independent interest (see also \cite{Katriel} for a similar result for $p(n)$, the number of partitions of $n$).

\begin{theorem}   \label{q(n) Katriel}
Suppose $\sum_{k=1}^r \alpha_k < \sum_{\ell=1}^s \beta_\ell$, where $\boldsymbol{\alpha} = \{\alpha_k\}_{k=1}^r$ and $\boldsymbol{\beta}= \{\beta_\ell\}_{\ell=1}^s$ are sequences of positive real numbers, and $r,s \in \mathbb N$.  Moreover, let $\{\mu_k\}_{k=1}^r, \{\nu_\ell\}_{\ell=1}^s \subset \mathbb N_0$, where $\mu_1 < \mu_2 < \cdots < \mu_r$ and $\nu_1 < \nu_2 < \cdots < \nu_s$.
Then there exists an $N$ (which depends on $\boldsymbol{\alpha},\boldsymbol{\beta}$, and $\mu_r$) such that for $n>N$,
\begin{align} \label{Katriel Equation}
    \sum_{k=1}^r \alpha_k \, q(n+\mu_k) \leq \sum_{\ell=1}^s \beta_\ell \, q(n+\nu_\ell).
\end{align}
Specifically, we may take {$N = {N_{\boldsymbol{\alpha},\boldsymbol{\beta},\mu_{r}}}$} as given in \eqref{eqn:Nmanysubscripts}.
\end{theorem} 

\begin{remark}
The assumptions $\mu_1 < \mu_2 < \cdots < \mu_r$ and $\nu_1 < \nu_2 < \cdots < \nu_s$ are made without loss of generality.
\end{remark}

The remainder of the paper is structured as follows. 
In Section \ref{sec_katriel}, we prove Theorem \ref{q(n) Katriel}, as well as an analogous result (Corollary \ref{rho(n,m) Katriel}) for $\rho(n,m),$ the number of partitions of $n$ into distinct parts at least $m$.
In Section \ref{sec_t=2} and \ref{sec_t=3} we prove Theorem \ref{diff} for $t=2$ and $t=3$, respectively.  
In Section \ref{Asymptotic section}, we establish asymptotic formulas for $a_t(n)$ and $b_t(n)$ for $t\in\{2,3\}$, proving Theorem \ref{limit}. In Section \ref{Further bias}, we establish bias results for additional partition statistics related to hook length.  Finally, in Section \ref{Further Conjectures},
we conjecture similar biases and asymptotics, as well as congruences for the number of hooks of fixed length in odd distinct partitions versus self-conjugate partitions.

\section{General linear inequalities for partitions into distinct parts}\label{sec_katriel}

The case of $t=3$ in the proof of Theorem \ref{diff} relies centrally on a certain linear partition inequality (see \eqref{eqn_rhoineqt3}). In this section, we prove results which immediately imply \eqref{eqn_rhoineqt3}. In fact, we prove much more general results, namely Theorem \ref{q(n) Katriel} and Corollary \ref{rho(n,m) Katriel}, which are also of independent interest. The most important idea behind the proof of Theorem \ref{q(n) Katriel} is the circle method developed by Hardy and Ramanujan in their seminal paper \cite{HardyRamanujan}. Using the modular transformation law of Dedekind's eta function, they prove the famous asymptotic formula
\begin{align*}
    p(n) \sim \dfrac{1}{4n\sqrt{3}} e^{\pi \sqrt{\frac{2n}{3}}}
\end{align*}
as $n \to \infty$. Their method was later refined by Rademacher \cite{Rad} to prove an exact formula for $p(n)$ expressed as a convergent infinite series involving Kloosterman sums and Bessel functions. Such formulas in the literature on partition theory are usually called {\it Rademacher-type formulas}. The  Rademacher-type formula for $q(n)$, established by Hagis in \cite{Hagis63}, is given by
\begin{equation}\label{q_Rad} q(n) = \frac{\pi}{(24n+1)^{1/2}} \mathop{\sum_{k=1}}_{k \text{ odd}}^\infty k^{-1} \Bigg( \ 
\sideset{_{}^{}}{_{}^{'}}\sum_{h \!\!\!\pmod{k}} \chi(h,k)e^{\frac{-2\pi i n h}{k}} \Bigg) I_1\Bigg(\frac{\pi}{12k}(48n+2)^{1/2}\bigg).
    \end{equation}
    Above, $I_1$ is the Bessel function of the first order (see \eqref{def_Isbessel}), $\chi$ is an explicit exponential function, and the sum on $h$ is taken over   $h\pmod{k}$ relatively prime to $k$. 
As discussed in the next subsection, this equation implies strong asymptotics for $q(n)$ which we use to prove Theorem \ref{q(n) Katriel}.
\subsection{Proof of Theorem \ref{q(n) Katriel}}
Recall that we aim to prove the inequality
\begin{align*}
    \sum_{k=1}^r \alpha_k q(n+\mu_k) \leq \sum_{\ell=1}^s \beta_\ell q(n+\nu_\ell),
\end{align*}
for sufficiently large $n$, where $\alpha_k, \beta_\ell \in \mathbb{R}^+$ and $\sum_{k=1}^r \alpha_k < \sum_{\ell=1}^s \beta_\ell$. We also explicitly define $N_{\boldsymbol{\alpha}, \boldsymbol{\beta}, \mu_r}$ such that this inequality holds for all $n >N_{\boldsymbol{\alpha}, \boldsymbol{\beta}, \mu_r}$.  

The main idea of the proof of Theorem \ref{q(n) Katriel} is that the Rademacher-type expansion \eqref{q_Rad} allows us to reduce linear inequalities for $q(n)$ to inequalities involving only Bessel functions.  Classical explicit asymptotics for Bessel functions then reduce this problem to elementary inequalities  involving exponential functions. To this end, we establish the following key proposition.

\begin{proposition} \label{q(n) Rademacher Bounds}
For $L \in \mathbb N$
and $\varepsilon >0,$ there exists an $N(\varepsilon,L)$ such that
$$0< \frac{q(n+L)}{q(n)} - 1 < \varepsilon$$ for $n> N(\varepsilon,L)$. Specifically, we may take $N(\varepsilon,L)=N_{A,B,C}(\varepsilon,L)$ as in \eqref{N-construction} below.  
\end{proposition}

\begin{remark}
    As will become clear in the proof of Proposition \ref{q(n) Rademacher Bounds}, we may choose any $A,B,C$ such that $A,B,C > 0$ and $0 < \frac{1}{A} + \frac{1}{B} + \frac{1}{C} < 1$ in defining $N(\varepsilon,L) = N_{A,B,C}(\varepsilon,L)$ in \eqref{N-construction}.  
    In practice, one can choose such $A,B,C$ to obtain a suitably small $N(\varepsilon,L)$.
\end{remark}

To prove Proposition \ref{q(n) Rademacher Bounds}, we will not need the full exact formula \eqref{q_Rad} for $q(n)$, but only certain error terms connected with it. Recall that  the Bessel functions $I_s(x)$, $s \in \C$,  are defined  by
\begin{align} \label{def_Isbessel}
    I_s(x) := \dfrac{1}{2\pi i} \int_{1 - i\infty}^{1 + i \infty} w^{-s-1} \exp\lp \dfrac{x}{2} \lp w + \dfrac 1w \rp \rp dw.
\end{align}
We will use the following result due to Beckwith and Bessenrodt, which is derived from Hagis' formula \eqref{q_Rad}.
\begin{theorem}[{\cite[Theorem 2.3]{BeckwithBessenrodt}}] \label{q(n) Exact Bound}
    Let $\mu =\mu_n:= \dfrac{\pi}{6 \sqrt{2}} \sqrt{24n+1}$. Then we have
    \begin{align*}
        q(n) = \dfrac{\pi^2}{6\sqrt{2}\mu} I_1(\mu) + E(\mu)
    \end{align*}
    where
    \begin{align*}
        \left| E(\mu) \right| \leq \dfrac{0.9 \pi^2}{6\sqrt{2}} \cdot \dfrac{e^\mu}{\mu^2}\lp 1 + 5 \mu^2 e^{-\mu} \rp.
    \end{align*}
\end{theorem}
We will also require effective estimates for the $I_1$ Bessel function.

\begin{proposition} \label{Bessel Bound}
For all $x > 3$, we have
\begin{align*}
    L_1(x) := \dfrac{e^x}{\sqrt{2\pi x}}\lp 1 - \dfrac{2}{x} \rp - \dfrac{2e^{-x}}{\sqrt{2\pi x}} < I_1(x) < \dfrac{e^x}{\sqrt{2\pi x}}\lp 1 + \dfrac{2}{x} \rp + \dfrac{2e^{-x}}{\sqrt{2\pi x}} =: U_1(x)
\end{align*} for $x>3$.
\end{proposition}
\begin{proof}

    Let $z$ be a complex number with $\left| \arg(z) \right| \leq \frac{\pi}{2}$ and let $s \in \C$. Then, from {\cite[Exercise 7.13.2]{Olver}}, 
    \begin{align*}
        I_1(z) = \dfrac{e^z}{\sqrt{2\pi z}} \left[ 1 +  \delta(z) \right] + i \dfrac{e^{-z}}{\sqrt{2\pi z}} \left[ 1 + \gamma(z) \right],
    \end{align*}
    with bounds
    \begin{align*}
         \left| \delta(z) \right| \leq \dfrac{3\pi}{8|z|} \exp\lp \dfrac{3\pi}{8|z|} \rp, \ \ \ \left| \gamma(z) \right| \leq \dfrac{3}{4|z|} \exp\lp \dfrac{3}{4|z|} \rp.
    \end{align*}
    Thus for $x > 0$, by the triangle inequality we have that
\begin{align*}
    I_1(x) \geq \dfrac{e^x}{\sqrt{2\pi x}}\lp 1 - \dfrac{3\pi}{8x} \exp\lp \dfrac{3\pi}{8x} \rp \rp - \dfrac{e^{-x}}{\sqrt{2\pi x}}\lp 1 + \dfrac{3}{4x} \exp\lp \dfrac{3}{4x} \rp \rp
\end{align*}
and
\begin{align*}
    I_1(x) \leq \dfrac{e^x}{\sqrt{2\pi x}}\lp 1 + \dfrac{3\pi}{8x} \exp\lp \dfrac{3\pi}{8x} \rp \rp + \dfrac{e^{-x}}{\sqrt{2\pi x}}\lp 1 + \dfrac{3}{4x} \exp\lp \dfrac{3}{4x} \rp \rp.
\end{align*}
It is easily checked that for all $x>3$ we have both $$\frac{3\pi}{8x} \exp\lp \frac{3\pi}{8x} \rp < \frac{2}{x}$$ and $$1 + \frac{3}{4x} \exp\lp \frac{3}{4x} \rp < 2$$  from which the statement of the proposition follows.  
\end{proof}

\subsubsection{Proof of Proposition \ref{q(n) Rademacher Bounds}}

Let $\varepsilon > 0$ and let $L \geq 1$ be an integer. We seek to find a real number $N(\varepsilon,L)$ such that for $n > N(\varepsilon,L)$, we  have
\begin{align*}
   0< \dfrac{q(n+L)}{q(n)} - 1 < \varepsilon.
\end{align*} Note that $0< q(n+L)/q(n) - 1$ for all $n\geq 5$. 
Let $\tilde \mu:=\mu_{n+L} = \frac{\pi}{6 \sqrt{2}} \sqrt{24(n+L)+1}$.
By Theorem \ref{q(n) Exact Bound},  it  suffices to prove that for sufficiently large $n$
\begin{align} \label{Halfway simple}
    \dfrac{\frac{\pi^2}{6\sqrt{2}\tilde \mu} I_1\lp \tilde \mu \rp + \widehat E(\tilde \mu)}{\frac{\pi^2}{6\sqrt{2} \mu} I_1(\mu) - \widehat E(\mu)} < 1 + \varepsilon,
\end{align}
where 
\begin{align*}
    \widehat E(\mu) := \dfrac{0.9 \pi^2}{6\sqrt{2}} \cdot \dfrac{e^\mu}{\mu^2}\lp 1 + 5 \mu^2 e^{-\mu} \rp.
\end{align*} 
It is straightforward to show that
\begin{align} \label{E-simplification}
    \widehat E(\mu) < \dfrac{11}{10} \cdot \dfrac{e^\mu}{\mu^2}
\end{align}
for $\mu > 9$, or for $n \geq 25$. Note that $\tilde \mu > \mu$. By applying these simplifications to \eqref{Halfway simple}, together with Proposition \ref{Bessel Bound}, it now suffices to find $N(\varepsilon,L)$ such that for $n > N(\varepsilon,L)$ we have 
\begin{align}\label{ineq2}
    \dfrac{\frac{\pi^2}{6\sqrt{2}\tilde \mu} U_1(\tilde \mu) + \frac{11 e^{\tilde \mu}}{10 \tilde \mu^2}}{\frac{\pi^2}{6\sqrt{2}\mu} L_1(\mu) - \frac{11 e^\mu}{10\mu^2}} < 1 + \varepsilon.
\end{align}
We assume that $n\geq 26$ so that we may apply \eqref{E-simplification} and so that the denominator in \eqref{ineq2} is positive.

Now, using the definitions of $L_1$ and $U_1$, we  simplify and show that \eqref{ineq2} is equivalent to 
\begin{align} \label{ineq3}
    \dfrac{\pi^{\frac 32} e^{\tilde \mu}}{12 \tilde \mu^{\frac 32}} + \dfrac{\pi^{\frac 32} e^{\tilde \mu}}{6 \tilde \mu^{\frac 52}} + \dfrac{\pi^{\frac 32} e^{-\tilde \mu}}{6 \tilde \mu^{\frac 32}} + \dfrac{11 e^{\tilde \mu}}{10 \tilde \mu^2} + \lp 1 + \varepsilon \rp \lp \dfrac{11 e^\mu}{10 \mu^2} + \dfrac{\pi^{\frac 32} e^\mu}{6 \mu^{\frac 52}} + \dfrac{\pi^{\frac 32} e^{-\mu}}{6 \mu^{\frac 32}} \rp < \lp 1 + \varepsilon \rp \dfrac{\pi^{\frac 32} e^\mu}{12 \mu^{\frac 32}}.
\end{align}

Since $e^{-\mu} \mu^{- \frac 32}$ is a decreasing function of $\mu$, the fact that $\tilde \mu \geq \mu$ implies
\begin{align*}
    \dfrac{\pi^{\frac 32} e^{-\tilde \mu}}{6 \tilde \mu^{\frac 32}} + \lp 1 + \varepsilon \rp  \dfrac{\pi^{\frac 32} e^{-\mu}}{6 \mu^{\frac 32}} \leq \lp 2 + \varepsilon \rp \dfrac{\pi^{\frac 32} e^{-\mu}}{6 \mu^{\frac 32}},
\end{align*}
and likewise the fact that $e^\mu \mu^{- \alpha}$ is increasing for $\mu > \alpha$ implies
\begin{align*}
    \dfrac{\pi^{\frac 32} e^{\tilde \mu}}{6 \tilde \mu^{\frac 52}} + \lp 1 + \varepsilon \rp \dfrac{\pi^{\frac 32} e^\mu}{6 \mu^{\frac 52}} \leq \lp 2 + \varepsilon \rp \dfrac{\pi^{\frac 32} e^{\tilde \mu}}{6 \tilde \mu^{\frac 52}}
\end{align*}
and
\begin{align*}
    \dfrac{11 e^{\tilde \mu}}{10 \tilde \mu^2} + \lp 1 + \varepsilon \rp \dfrac{11 e^\mu}{10 \mu^2} \leq \lp 2 + \varepsilon \rp \dfrac{11 e^{\tilde \mu}}{10 \tilde \mu^2}.
\end{align*}
Therefore, to prove \eqref{ineq3} we need only prove the inequality
\begin{align*}
    \dfrac{\pi^{\frac 32} e^{\tilde \mu}}{12 \tilde \mu^{\frac 32}} + \lp 2 + \varepsilon \rp \dfrac{\pi^{\frac 32} e^{\tilde \mu}}{6 \tilde \mu^{\frac 52}} + \lp 2 + \varepsilon \rp \dfrac{11 e^{\tilde \mu}}{10 \tilde \mu^2} + \lp 2 + \varepsilon \rp \dfrac{\pi^{\frac 32} e^{-\mu}}{6 \mu^{\frac 32}} < \lp 1 + \varepsilon \rp \dfrac{\pi^{\frac 32} e^\mu}{12 \mu^{\frac 32}}.
\end{align*}
Define functions $F_j(n)$, $1 \leq j \leq 4$, and $M(n)$ so that the above inequality reads
\begin{align*}
    F_1(n) + F_2(n) + F_3(n) + F_4(n) < M(n).
\end{align*}
$F_1(n)$ is asymptotically larger than $F_j(n)$ for $2 \leq j \leq 4$, and so it is natural to simplify further by bounding $F_j(n)$, $2 \leq j \leq 4$ with small multiples of $F_1(n)$. For any choice of positive real numbers $A, B, C$, it is straightforward to show that $F_2(n) < \frac{\varepsilon}{A} F_1(n)$ for all
\begin{align*}
    n > N_A = N_A(\varepsilon, L) := \frac{12 A^2 \lp 2 + \varepsilon \rp^2}{\pi^2 \varepsilon^2} - L - \frac{1}{24},
\end{align*}
that $F_3(n) < \frac{\varepsilon}{B} F_1(n)$ for all 
\begin{align*}
    n > N_B = N_B(\varepsilon, L) := \frac{910787328 B^4 \lp 2 + \varepsilon \rp^4}{10000 \pi^8 \varepsilon^4} - L - \frac{1}{24},
\end{align*}
and that $F_4(n) < \frac{\varepsilon}{C} F_1(n)$ for
\begin{align*}
    n > N_C = N_C(\varepsilon, L) := \frac{3}{4\pi^2}\log^2\lp \frac{2C(2+\varepsilon)(1+L)}{\varepsilon} \rp - \frac{1}{24}.
\end{align*}
The calculation of $N_C$ above uses the assumptions $n \geq 1$ and $\tilde \mu \geq \mu$. We have reduced our problem to proving that
\begin{align} \label{Reduced EQ}
    \lp 1 + D \varepsilon \rp \dfrac{\pi^{\frac 32} e^{\tilde \mu}}{12 \tilde \mu^{\frac 32}} < \lp 1 + \varepsilon \rp \dfrac{\pi^{\frac 32} e^\mu}{12 \mu^{\frac 32}}
\end{align}
for sufficiently large $n$, where 
$$D = D(A,B,C):=\frac{1}{A}+\frac{1}{B}+\frac{1}{C}.$$
Because the choice of $A,B,C$ is made freely, we make these choices so that $0<D<1$, which is required for \eqref{Reduced EQ} to hold for sufficiently large $n$.

Let $\tilde n := 24n+1$ for convenience.  By elementary manipulations, \eqref{Reduced EQ} is equivalent to
\begin{align*}
    \lp \dfrac{\tilde n}{\tilde n + 24L} \rp^{\frac 34} e^{\frac{\pi}{6\sqrt{2}}\left[ \sqrt{\tilde n + 24L} - \sqrt{\tilde n} \right]} < \dfrac{1+\varepsilon}{1 + D\varepsilon }.
\end{align*}
Since $\tilde n > 0$ and $L > 0$, we have $\frac{\tilde n}{\tilde n + 24L} < 1$ and $\sqrt{\tilde n + 24L} - \sqrt{\tilde n} = \frac{24L}{\sqrt{\tilde n + 24L} + \sqrt{\tilde n}} \leq \frac{12L}{\sqrt{\tilde n}}$. Thus, it is enough to prove that for large enough $n$ we have
\begin{align*}
    e^{\frac{\sqrt{2} L \pi}{\sqrt{\tilde n}}} < \dfrac{1+\varepsilon}{1 + D\varepsilon },
\end{align*}
which holds if and only if
 $$  n>
 N_D = N_{D}(\varepsilon,L):=\frac{ L^2 \pi^2}{{12}\log^2 \left( \frac{\displaystyle 1+\varepsilon}{\displaystyle 1 +  \displaystyle \varepsilon \displaystyle D }\right)}-\frac{1}{24}. $$

Therefore,  we conclude that the inequality in Proposition \ref{q(n) Rademacher Bounds} holds for all $n > N_{A,B,C}(\varepsilon,L)$, where
\begin{align} \label{N-construction}
      N_{A,B,C}(\varepsilon,L) :=  \max\lp N_A, N_B, N_C,N_D, 26 \rp.
\end{align}

\subsubsection{Proof of Theorem \ref{q(n) Katriel}} 

Since $q(n)>0$ for all $n\geq 0$, 
\eqref{Katriel Equation} is equivalent to \begin{align}\label{rewrite}
    \sum_{k=1}^r \alpha_k \dfrac{q\lp n + \mu_k \rp}{q(n)} \leq \sum_{\ell = 1}^s \beta_\ell \dfrac{q\lp n + \nu_\ell \rp}{q(n)}.
\end{align} 
Moreover, since for all $n\geq 0$, we have $q(n+\nu_\ell)/q(n)\geq 1$ for all $1\leq \ell\leq s$ and $q(n+\mu_k)\leq q(n+\mu_r)$ for all $1\leq k \leq r$, inequality \eqref{Katriel Equation}  holds for all $n$ such that \begin{align*}
    \lp \sum_{k=1}^r \alpha_k \rp \lp \dfrac{q\lp n + \mu_r \rp}{q(n)} - 1 \rp \leq \sum_{\ell = 1}^s \beta_\ell - \sum_{k=1}^r \alpha_k.
\end{align*}
Let $L:=\mu_r$ and \begin{align}\label{eqn_ep19}
    \varepsilon := \lp \sum_{k=1}^r \alpha_k \rp^{-1} \lp \sum_{\ell = 1}^s \beta_\ell - \sum_{k=1}^r \alpha_k \rp.
\end{align}

Choose $A,B,C>0$ such that $\frac{1}{A}+\frac{1}{B}+\frac{1}{C}<1$ and define
\begin{equation}\label{eqn:Nmanysubscripts}
N_{\boldsymbol{\alpha},\boldsymbol{\beta},\mu_r}:= N_{A,B,C}(\varepsilon, L)
\end{equation}
where  $N_{A,B,C}(\varepsilon, L)$ is given by \eqref{N-construction} for the above choice of $\varepsilon$ and $L$. The proof then follows from Proposition \ref{q(n) Rademacher Bounds}. \qed

\bigskip

\begin{remark} We point out that in practice, it may be possible to improve (i.e. decrease) the bound $N_{\boldsymbol{\alpha},\boldsymbol{\beta},\mu_r}$ obtained in the proof of Theorem \ref{q(n) Katriel} above.  For example, at the outset, any summands $\alpha_k q(n+\mu_k)$ and $\beta_\ell q(n+\nu_\ell)$ appearing in an inequality \eqref{Katriel Equation} such that $\alpha_k = \beta_\ell$ and $\mu_k = \nu_\ell$ can be canceled, producing an equivalent inequality with a larger $\varepsilon$ in \eqref{eqn_ep19}.  Similarly, in case of any summands $\alpha_k q(n+\mu_k)$ and $\beta_\ell q(n+\nu_\ell)$ with $\alpha_k\neq \beta_\ell$ but $\mu_k = \nu_\ell$, one may subtract $\min(\alpha_k,\beta_\ell) q(n+\mu_k)$ from both sides of \eqref{Katriel Equation} yielding an equivalent inequality with  potentially   smaller   $L$ and larger   $\varepsilon$ as chosen in the proof of Theorem \ref{q(n) Katriel} above.  Further, given $L, \varepsilon$ as in the proof of Theorem \ref{q(n) Katriel}, one may seek to choose appropriate $A,B,$ and $C$   such that $N_{\boldsymbol{\alpha},\boldsymbol{\beta},\mu_r}$ in \eqref{eqn:Nmanysubscripts} is minimized. \end{remark}

\subsection{General linear inequalities for distinct partitions without small parts}

In our proof of Theorem \ref{diff} for $t=3$, we require a linear partition inequality for the $\rho(n,m)$, which counts partitions of $n$ into distinct parts all at least $m$. To prove this inequality, we make use of the following result of  Erd\H{o}s-Nicolas-Szalay \cite{ENS} which relates $q(n)$ and $\rho(n,m)$.

\begin{theorem}[Theorem 1, \cite{ENS}] \label{thm_ENS} For all $n$ and $m$ satisfying $1\leq m \leq n,$ we have that 
\begin{align}\label{eqn_erdthm} \frac{q(n)}{2^{m-1}} \leq \rho(n,m) \leq \frac{q(n + m(m-1)/2)}{2^{m-1}}.
\end{align}
\end{theorem}

General linear inequalities for $\rho(n,m)$ then follow from Theorems \ref{q(n) Katriel} and \ref{thm_ENS}. We make this explicit in the following corollary.

\begin{corollary} \label{rho(n,m) Katriel} {
Suppose $\sum_{k=1}^r \alpha_k < \sum_{\ell=1}^s \beta_\ell$, where $\boldsymbol{\alpha} =\{\alpha_k\}_{k=1}^r$ and $\boldsymbol{\beta} = \{\beta_\ell\}_{\ell=1}^s$ are sequences of positive real numbers, and $r,s \in \mathbb N$.  Moreover, let $\{\mu_k\}_{k=1}^r, \{\nu_\ell\}_{\ell=1}^s \subset \mathbb N_0$, where $\mu_1 < \mu_2 < \cdots < \mu_r$ and $\nu_1 < \nu_2 < \cdots < \nu_s$.
Then for any $m\in \mathbb N$, there exists an $N$ (which depends on $\boldsymbol{\alpha},\boldsymbol{\beta},\mu_r,$ and $m$) such that for $n>N$,
\begin{align} \label{rho Katriel Equation}
    \sum_{k=1}^r \alpha_k \rho(n+\mu_k,m) \leq \sum_{\ell=1}^s \beta_\ell \rho(n+\nu_\ell,m).
\end{align}  Specifically, we may take $N=N_{\boldsymbol{\alpha},\boldsymbol{\beta},\mu_r + m(m-1)/2}$ as in \eqref{eqn:Nmanysubscripts}.}

\end{corollary}

\begin{remark} When $m=1$, Corollary \ref{rho(n,m) Katriel} becomes Theorem \ref{q(n) Katriel}.
\end{remark}

\begin{proof}[{Proof of Corollary \ref{rho(n,m) Katriel}}]
We apply Theorem \ref{thm_ENS} to each $\rho(n+\mu_k,m)$   and obtain
\begin{align}\label{eqn_rhoq1}
 \sum_{k=1}^r \alpha_k \rho(n+\mu_k,m) \leq 2^{1-m} \sum_{k=1}^r \alpha_k q(n+\mu_k+ m(m-1)/2). 
\end{align}
By Theorem \ref{q(n) Katriel}, after replacing $\{\mu_k\}$ with $\{\mu_k +m(m-1)/2\}$, the right-hand side of \eqref{eqn_rhoq1} is at most
\begin{align}\label{eqn_rhoq2} 2^{1-m} \sum_{\ell=1}^s \beta_{\ell} q(n+\nu_\ell)
\end{align} for $n>N_{\boldsymbol{\alpha},\boldsymbol{\beta},\mu_r + m(m-1)/2}$.
Applying Theorem \ref{thm_ENS} again, this time to each $q(n+\nu_\ell)$, we obtain that \eqref{eqn_rhoq2} is at most 
$$\sum_{\ell=1}^s \beta_\ell\rho(n+\nu_\ell,m)$$ for $n>N_{\boldsymbol{\alpha},\boldsymbol{\beta},\mu_r + m(m-1)/2}$
as desired.
\end{proof}

\section{Hooks of length $2$}\label{sec_t=2}

In this section, we prove Theorem \ref{diff} in the case of hooks of length $t=2$.  To do so, we first establish  relevant generating functions   for odd (respectively distinct) partitions in Section \ref{sec_t2odd} (respectively Section \ref{sec_t2dist}).  In Section \ref{sec_t2diff} we use these results to prove Theorem \ref{diff} for $t=2$. We also give a combinatorial interpretation of the excess $a_2(n)-b_2(n)$, in analogy with Theorem \ref{AB}.

First we introduce some notation  used throughout the remainder of the article. The \emph{$q$-Pochhammer symbol} is defined for $n \in \mathbb N_0 \cup \{\infty\}$ by 
$$(a;q)_n := \prod_{j=0}^{n-1} (1-a q^j).$$  We will make use of the following well-known partition generating functions (with $|q|<1$): 
\begin{align*}
\sum_{n=0}^\infty p(n) q^n &= \frac{1}{(q;q)_\infty}, \ \ \ \ 
\sum_{n=0}^\infty q(n) q^n = {(-q;q)_\infty},\smallskip \\
\sum_{n=0}^\infty o(n) q^n &= \frac{1}{(q;q^2)_\infty}, \ \ \ 
\sum_{n=0}^\infty do(n) q^n = {(-q;q^2)_\infty},
\end{align*} where
$o(n)$ equals the number of odd parts partitions of $n$, and $do(n)$ equals the number of distinct odd parts partitions of $n$. 
Since the only partition of $n=0$ is the empty partition, we have $p(0)=q(0)=o(0)=do(0)=1$.

We may abuse notation, and use partition and   Young  diagram interchangeably;  we do the same for parts of a partition and rows of its diagram.  

\subsection{Odd partitions and hooks of length $\mathbf{t=2}$}\label{sec_t2odd}

We establish the generating function for $a_2(n)$. 

\begin{proposition} \label{gf_a2} We have $$\sum_{n\geq 0}a_2(n)q^n=\frac{1}{(q;q^2)_\infty}\left(q^2+\sum_{n\geq 2}(q^{2n-1}+q^{2(2n-1)})\right).
$$
\end{proposition}

\begin{proof} Let $a_2(m,n)$  be the number of odd   partitions of $n$ with $m$ hooks of length $2$ and denote by $F_2(z;q)$ the bivariate generating function for the sequence $a_2(m,n)$, i.e., $$F_2(z;q):=\sum_{n,m\geq 0}a_2(m,n)z^mq^n.$$

If $\lambda$ is an odd partition, then by considering the possible ways a hook of length 2 can occur in a Young diagram, we see that the number of hooks of length $2$ in $\lambda$ is equal to the number of different part sizes of $\lambda$ that are greater than $1$ plus the number of different  parts of $\lambda$ that occur at least twice. Therefore we have
\begin{align} \label{eqn_F2zq} F_2(z;q)=\left(1+q+\frac{zq^2}{1-q}\right)\prod_{n=2}^\infty \left(1+zq^{2n-1}+\frac{z^2q^{2(2n-1)}}{1-q^{2n-1}}\right).\end{align}

From the definition of $a_2(m,n)$, we have $$\sum_{n\geq 0}a_2(n)q^n=\frac{\partial}{\partial z}\bigg|_{z=1}F_2(z;q).$$ Using logarithmic differentiation finishes the proof. 
\end{proof}

\subsection{Distinct partitions and hooks of length $\mathbf{t=2}$}\label{sec_t2dist}
We establish the generating function for $b_2(n)$. 

 \begin{proposition} \label{gf_b2} We have \begin{align}\label{eqn_b2}\sum_{n\geq 0}b_2(n)q^n=   \frac{q^2}{1-q}(-q^2;q)_\infty.
\end{align}
 \end{proposition}

 \begin{proof} Let $b_2(m,n)$ be the number of    distinct partitions of $n$ with $m$ hooks of length $2$ and denote by $G_2(z;q)$ the bivariate generating function for the sequence $b_2(m,n)$, i.e., $$G_2(z;q):= \sum_{n,m\geq 0} b_2(m,n)z^mq^n.$$

  In the Young diagram of a distinct partition $\lambda$, the number of hooks of length $2$ in $\lambda$ equals the number of parts $\lambda_i$ such that $\lambda_i-\lambda_{i+1}\geq 2$, where $\lambda_k=0$ if $k>\ell(\lambda$). 
  Thus, the number of hooks of length $2$ in $\lambda$ can be calculated as follows:
   start with the Young diagram of $\lambda$ and   
  remove the \textit{Sylvester triangle},
  i.e., subtract $1$ from the last part, $2$ from the second to last part, etc., to obtain an ordinary partition $\mu$. Then, we count the number of different part sizes in~$\mu$. We note that the number of different part sizes in $\mu$ is equal to the number of different part sizes in its conjugate $\mu'$. Let $u_n(t,m)$ be the number of partitions of $m$ with $t$ different part sizes and all parts at most $n$. Then  \begin{align}\label{eqn_G2zq} G_2(z;q)=\sum_{n\geq 1}q^{n(n+1)/2} \left(\sum_{m,t\geq 0} u_n(t,m) z^tq^m\right)=\sum_{n\geq 1}q^{n(n+1)/2}\prod_{j=1}^n\left(1+\frac{zq^j}{1-q^j}\right).\end{align}
 
 From the definition of $b_2(m,n)$, we have $$\sum_{n\geq 0}b_2(n)q^n=\frac{\partial}{\partial z}\bigg|_{z=1}G_2(z;q).$$ Logarithmic differentiation gives
\begin{equation*}\label{G2-dist}\frac{\partial}{\partial z}\Big|_{z=1}G_2(z;q)=\sum_{n\geq 1}q^{n(n+1)/2}\frac{1}{(q;q)_n}\frac{q-q^{n+1}}{1-q}.
\end{equation*} Using  the well-known limiting case of the $q$-binomial theorem (see e.g. \cite[(2.2.6)]{And}),  
\begin{equation}\label{q-binomial}\sum_{n=0}^\infty  \frac{q^{\frac{n(n+1)}{2}}z^n}{(q;q)_n}= (-zq;q)_\infty\end{equation} with $z=q$, we obtain 
\begin{align*}\sum_{n\geq 0}b_2(n)q^n= & \frac{q}{1-q}\sum_{n\geq 1} \frac{q^{\frac{n^2+n}{2}}}{(q;q)_{n-1}}   =  
\frac{q^2}{1-q}\sum_{n\geq 0} \frac{q^{\frac{n^2+3n}{2}}}{(q;q)_{n}}  =   \frac{q^2}{1-q}(-q^2;q)_\infty.
\end{align*}\end{proof}

\subsection{Proof of Theorem \ref{diff} for hooks of length $\mathbf{t=2}$} \label{sec_t2diff}
Using  Propositions \ref{gf_a2} and \ref{gf_b2},  we have\begin{align*}
\notag \sum_{n\geq 0}& (a_2(n)-b_2(n))q^n =\frac{1}{(q;q^2)_\infty}\left(q^2+\sum_{n\geq 2}(q^{2n-1}+q^{2(2n-1)})\right)- \frac{q^2}{1-q}(-q^2;q)_\infty. \end{align*} Then, using geometric series, Euler's identity \cite[(1.2.5)]{And}, $1/(q;q^2)_\infty=(-q;q)_\infty$, and straightforward manipulations, we obtain \begin{align} \notag \sum_{n\geq 0} (a_2(n)-b_2(n))q^n & =  \frac{1}{(q;q^2)_\infty} \frac{q^2(1+q+q^3)}{1-q^4}- \frac{q^2}{1-q}\frac{(-q;q)_\infty}{1+q}\\ \notag &
=(-q;q)_\infty\frac{q^2}{1-q^2}\left( \frac{1+q+q^3}{1+q^2}-1\right)\\ \notag &
 =(-q;q)_\infty\frac{q^2}{1-q^2}\left( \frac{q-q^2+q^3}{1+q^2}\right) \\ \label{t2diff} 
 &= q^3\frac{1+q^3}{1-q^2}(-q^3;q)_\infty. \end{align} 
 Clearly, the expression in 
 \eqref{t2diff}, expanded as a $q$-series, has non-negative coefficients. This concludes the proof of Theorem \ref{diff} for $t=2$.\qed

 Next we give a combinatorial interpretation for $a_2(n)-b_2(n)$. Here and throughout, for $i \in \mathbb{Z}^+$, we define the multiplicity $m_\lambda(i)$ of $i$ in $\lambda$ to be the number of times  $i$ appears as a part in partition $\lambda$.

\begin{proposition} For $n\geq 0$, $a_2(n)-b_2(n)=w(n)$, where $w(n)$ is the total number of different part sizes greater than $1$ in all odd partitions $\lambda$ of  $n$ with $m_\lambda(1)\equiv 0,3 \pmod 4$.
\end{proposition}

\begin{proof} Let $w(m,n)$ be the number of odd partitions $\lambda$ of  $n$ with $m_\lambda(1)\equiv 0,3 \pmod 4$ and exactly $m$ different part sizes greater than $1$. Then, 


$$H(z;q):=\sum_{n,m\geq 0}w(m,n)z^mq^n=\left(\frac{1}{1-q^4}+ \frac{q^3}{1-q^4}\right)\prod_{k=1}^\infty \left(1+\frac{zq^{2k+1}}{1-q^{2k+1}} \right).$$ 

Using this, we compute

\begin{align}\notag\sum_{n\geq0}w(n)q^n=\frac{\partial}{\partial z}\Big|_{z=1}H(z;q)& =\frac{1+q^3}{1-q^4} \prod_{k=1}^\infty \left(1+\frac{q^{2k+1}}{1-q^{2k+1}} \right)\sum_{k=1}^\infty q^{2k+1}\\ \notag  & = \frac{1+q^3}{1-q^4}\cdot \frac{1}{(q^3;q^2)_\infty}\cdot \frac{q^3}{1-q^2} \\ \notag  & = q^3\frac{1+q^3}{1-q^2}\cdot \frac{1-q}{1-q^4}\cdot(-q;q)_\infty\\ \notag  & = q^3\frac{1+q^3}{1-q^2}\cdot(-q^3;q)_\infty,\end{align} 

which, together with \eqref{t2diff}, completes the proof.
\end{proof}

\section{Hooks of length $3$} \label{sec_t=3}
In this section, we prove Theorem \ref{diff} in the case of hooks of length $t=3$.  To do so, we first establish  relevant generating functions   for odd (respectively distinct) partitions in Section \ref{sec_t3odd} (respectively Section \ref{sec_t3dist}).  In Section \ref{sec_t3diff} we use these results as well as results from Section \ref{sec_katriel} to prove Theorem \ref{diff} for $t=3$.   

 First, we introduce some useful notation.  For any partition $\lambda$, denote by $\ell(\lambda)$ the length of $\lambda$, i.e., the number of parts in $\lambda$. We denote by  $\ell_1(\lambda)$ (respectively $\ell_2(\lambda)$)   the number of parts $\lambda_i$ of $\lambda$ with $\lambda_i-\lambda_{i+1}=1$ (respectively $\lambda_i-\lambda_{i+1}=2$).  We assume $\lambda_k=0$ if $k>\ell(\lambda)$.

\subsection{Odd partitions and hooks of length $\mathbf{t=3}$}\label{sec_t3odd} 
We establish the generating function for $a_3(n).$
\begin{proposition}\label{prop_F3gf}
    We have that
    $$  \sum_{n\geq 0} a_3(n) q^n =  (-q^3;q)_\infty \frac{q^3(1+q^3)}{1-q^2} + 
(-q;q)_\infty\left(\frac{q^6}{1-q^4}+\frac{q^3}{1-q^6}\right).$$
\end{proposition}
\begin{proof} 
Let $\lambda$ be an odd partition.  Among parts equal to $1$,  there is a hook of length $3$ only in the third to last part equal to $1$ (thus, only if the multiplicity of $1$ is at least three).  For all other parts, there is a hook of length 3 in the second to last and third to last occurrence of the part size (if the multiplicity of the part permits). There is also a hook of length 3 in the last occurrence of the part size if that last occurrence is in row~$i$ with $\lambda_i-\lambda_{i+1}\neq 2$. For example, the hooks of length $3$ are marked in the Young diagram of $(7, 7, 7, 7, 3, 3,1,1,1,1)$. $${\footnotesize 
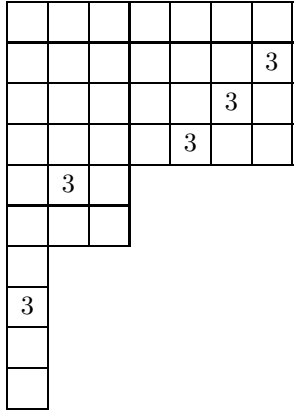
\begin{ytableau} \ & \ & \ &\ &\ &\ &\ \\ \ & \ & \ &\ &\ &\ & 3 \\ \ & \ & \ &\ &\ &3 &\ \\ \ & \ & \ &\ &3 &\ &\ \\ \ & 3& \ \\ \ & \ & \ \\ \ \\ 3 \\ \ \\ \ \end{ytableau}}$$

Let $a_3(m,n)$   be the number of odd    partitions of $n$ with $m$ hooks of length $3$. It follows from above that $a_3(m,n)$ is the number of odd partitions $\lambda$ of $n$ such that $$\delta_{m_\lambda(1)\geq 3}+\bigg(\sum_{u\geq 2}(\delta_{m_\lambda(u)\geq 1}+\delta_{m_\lambda(u)\geq 2}+\delta_{m_\lambda(u)\geq 3})\bigg)-\ell_2(\lambda)=m.$$ Here and throughout, $\delta_{\rho}$ denotes the Kronecker delta symbol, which evaluates to $1$ if property $\rho$ is true, and $0$ if not.  

If   $x(k,n)$  is the number of odd partitions $\lambda$ of $n$ such that $$\delta_{m_\lambda(1)\geq 3}+\sum_{u\geq 2}(\delta_{m_\lambda(u)\geq 1}+\delta_{m_\lambda(u)\geq 2}+\delta_{m_\lambda(u)\geq 3})=k,$$  and $y(k,n)$ is the number of odd partitions $\lambda$ of $n$ with $\ell_2(\lambda)=k$, then $$a_3(n)=\sum_{m\geq 0}m\, a_3(m,n)= \sum_{k\geq 0}k\, x(k,n) - \sum_{k\geq 0}k\, y(k,n).$$ 
Let $\displaystyle \mathcal F_1(z;q): =\sum_{n,k\geq 0} x(k,n)z^kq^n$ and $\displaystyle \mathcal F_2(z;q):=\sum_{n,k\geq 0} y(k,n)z^{k}q^n$. Then, $$\sum_{n\geq 0}  a_3(n) q^n =\frac{\partial}{\partial z}\Big|_{z=1}\mathcal{F}_1(z;q)-\frac{\partial}{\partial z}\Big|_{z=1}\mathcal{F}_2(z;q).$$
By standard arguments, we have that
\begin{align*}
\mathcal F_1(z;q)=\left(1+q+q^2+\frac{zq^3}{1-q}\right)\prod_{n\geq 1}\left(1+zq^{2n+1}+ z^2q^{2(2n+1)} +\frac{z^3q^{3(2n+1)}}{1-q^{2n+1}}\right).\end{align*} 
To find $\mathcal F_2(z;q)$,  we consider  the conjugate of the $2$-modular diagram of an odd partition. The $2$-{modular diagram} of an odd partition $\lambda$ is a Young diagram in which row $i$ has $ \frac{\lambda_i+1}{2}$ boxes, with the first box filled with $1$, and the remaining boxes filled with $2$. Then the sum of the entries in row~$i$ equals $\lambda_i$. 
 See the left diagram in Figure \ref{fig:2modfig2}.
 
Given an odd partition $\lambda$ of $n$, if a part $\lambda_i$ satisfies $\lambda_i-\lambda_{i+1}=2$, then the $i$-th row in the $2$-modular diagram of $\lambda$ has last box filled with $2$ and is exactly one box longer than the next row. To find $\ell_2(\lambda)$, we conjugate the $2$-modular diagram of $\lambda$, and in it 
 find the number of   rows with multiplicity $1$ among the rows filled with $2$ that are strictly shorter than the first row. For example, 
the $2$-modular diagram of $\lambda=(11,9,5,3)$ with marked relevant rows  and its conjugate are shown in Figure \ref{fig:2modfig2}. Here, $\ell_2(\lambda)=2$. 

\begin{figure}[H]
	\centering
    {\footnotesize $\begin{ytableau} 1&2&{2}&2&2&{\color{red}\mathbf{2}}\\1&2&{2}&2&2\\1&2&{\color{red}\mathbf{2}}\\1&2\end{ytableau}$ \qquad \qquad  $\begin{ytableau} 1&1&1&1\\2&2&2&2\\2&2&{\color{red}\mathbf{2}}\\2&2\\2&2\\{\color{red}\mathbf{2}}\end{ytableau}$}
    \caption{2-modular diagram and conjugate for $\lambda=(11,9,5,3)$.}
    \label{fig:2modfig2}
\end{figure}

Thus, $$\mathcal{F}_2(z;q)=\sum_{n\geq 1} \frac{q^n}{1-q^{2n}}\prod_{j=1}^{n-1}\left(1+zq^{2j}+\frac{q^{2(2j)}}{1-q^{2j}}\right).$$

We use logarithmic differentiation to  obtain
\begin{align*}\frac{\partial}{\partial z}\Big|_{z=1}\mathcal F_1(z;q)& =\frac{1}{(q;q^2)_\infty}\left(q^3+\sum_{n\geq 1}(q^{2n+1}+q^{2(2n+1)}+ q^{3(2n+1)}) \right)\\ &  =(-q;q)_\infty\left(q^3+ \frac{q^3}{1-q^2}+\frac{q^6}{1-q^4}+\frac{q^9}{1-q^6}\right)
\end{align*}
and 
\begin{align*}\frac{\partial}{\partial z}\Big|_{z=1}\mathcal F_2(z;q)& =\sum_{n\geq 1} q^n\frac{1}{(q^2;q^2)_n}\sum_{j=1}^{n-1}q^{2j}(1-q^{2j}). \end{align*} Using the well-known identity $\displaystyle\sum_{n\geq 0}  {z^n}/{(q;q)_n}={1}/{(z;q)_\infty}$ (see e.g. \cite[(2.2.5)]{And})  and the geometric series identity, after simplifying, we further
obtain that 
 \begin{align*}\frac{\partial}{\partial z}\Big|_{z=1}\mathcal F_2(z;q)& = \frac{1}{(q;q^2)_\infty}\left(\frac{q^2}{1-q^2}-\frac{q^4}{1-q^4}+\frac{(1+q)(1+q^3)}{1-q^4}-\frac{1+q}{1-q^2} \right)\\ &  =\frac{1}{(q;q^2)_\infty}\frac{q^4}{1-q^4}. \end{align*}
 
Finally, we have
\begin{align} \sum_{n\geq 0}  a_3(n) q^n &= \frac{\partial}{\partial z}\Big|_{z=1}\mathcal{F}_1(z;q)-\frac{\partial}{\partial z}\Big|_{z=1}\mathcal{F}_2(z;q) \notag \\ &=(-q;q)_\infty\left(q^3+ \frac{q^3}{1-q^2}+\frac{q^6}{1-q^4}+\frac{q^9}{1-q^6} - \frac{q^4}{1-q^4}\right) \notag \\
&=  (-q^3;q)_\infty \frac{q^3(1+q^3)}{1-q^2} + 
(-q;q)_\infty\left(\frac{q^6}{1-q^4}+\frac{q^3}{1-q^6}\right). \label{eqn_genF3a3}
\end{align}
\end{proof}

\subsection{Distinct partitions and hooks of length $\mathbf{t=3}$}\label{sec_t3dist}  We establish the generating function for $b_3(n).$
\begin{proposition}\label{prop_G3gf}
    We have that
    $$ \sum_{n\geq 0} b_3(n) q^n = (-q;q)_\infty \sum_{m\geq 2}\frac{q^m}{1+q^m}-\frac{q^2}{1-q^2}(-q^3;q)_\infty .$$
\end{proposition}
\begin{proof}
 Let  $b_3(m,n)$  denote the number of distinct   partitions of $n$ with $m$ hooks of length $3$.  In a distinct partition $\lambda$, there is a hook of length $3$ in every row $\lambda_i>1$ except when $\lambda_i-\lambda_{i+1}=2$.  Thus,
$b_3(m,n)$ is the number of distinct partitions of $n$ such that $\ell(\lambda)-m_\lambda(1)-\ell_2(\lambda)=m$. 

If $u(k,n)$ is the number of distinct partitions of $n$ with exactly $k$ parts greater than $1$, and $v(k,n)$ is the number of distinct partitions $\lambda$ of $n$ with $\ell_2(\lambda)=k$, then $$b_3(n)=\sum_{m\geq 0}m\, b_3(m,n)= \sum_{k\geq 0}k\, u(k,n) - \sum_{k\geq 0}k\, v(k,n).$$

Let $\mathcal G_1(z;q):=\displaystyle\sum_{n,k\geq 0} u(k,n)z^kq^n$ and $\mathcal G_2(z;q):=\displaystyle\sum_{n,k\geq 0} v(k,n)z^{k}q^n$. Then, 
\begin{align} \label{eqn_b3gfdiff} \sum_{n\geq 0}b_3(n)q^n =\frac{\partial}{\partial z}\Big|_{z=1} \mathcal{G}_1(z;q)-\frac{\partial}{\partial z}\Big|_{z=1}\mathcal{G}_2(z;q).\end{align}
We have  $$\mathcal G_1(z;q)=(1+q)(-zq^2;q)_\infty.$$ 

To find $\mathcal{G}_2(z;q)$, let $\lambda$ be a partition with distinct parts and, as in the case $t=2$, remove the Sylvester triangle from $\lambda$
to obtain a partition $\mu$ with at most $\ell(\lambda)$ parts. Note that parts $\mu_i$ in $\mu$ such that $\mu_i-\mu_{i+1}=1$ correspond to parts $\lambda_i$ in $\lambda$ such that $\lambda_i-\lambda_{i+1}=2$. Thus, $\ell_1(\mu)=\ell_2(\lambda)$. For a partition $\nu$, denote by $\widetilde\ell(\nu)$ the number of parts of $\nu$ with multiplicity one. 
Since $\ell_1(\mu)=\widetilde\ell(\mu')$, we have

$$\mathcal G_2(z;q)=\sum_{m\geq 1} q^{\frac{m(m+1)}{2}}\prod_{j=1}^m \left(1+zq^j + \frac{q^{2j}}{1-q^j}\right).$$

Using logarithmic differentiation, we compute
\begin{align}\label{eqn_dG1zq} \frac{\partial}{\partial z}\Big|_{z=1}\mathcal G_1(z;q)=(-q;q)_\infty \sum_{m\geq 2}\frac{q^m}{1+q^m} \end{align}
and 
\begin{align} \notag 
   \frac{\partial}{\partial z}\Big|_{z=1}   \mathcal G_2(z;q) & =   \sum_{n\geq 1} \frac{q^{\frac{n(n+1)}{2}}}{(q;q)_n} \sum_{j=1}^n q^j(1-q^j) \\ \notag & =  \sum_{n\geq 1} \frac{q^{\frac{n(n+1)}{2}}}{(q;q)_n}\left(\frac{1-q^{n+1}}{1-q}-\frac{1-q^{2(n+1)}}{1-q^2}\right)\\ \notag & = \frac{q}{1-q^2}\sum_{n\geq 1} \frac{q^{\frac{n(n+1)}{2}}}{(q;q)_{n-1}}(1-q^{n+1}) \\ \notag  & = \frac{q}{1-q^2}\left(\sum_{n\geq 0} \frac{q^{\frac{n(n+1)}{2}+n+1}}{(q;q)_n}-\sum_{n\geq 0} \frac{q^{\frac{n(n+1)}{2}+2n+3}}{(q;q)_n}\right)\\ \notag & = \frac{q}{1-q^2}\left(q(-q^2;q)_\infty-q^3(-q^3;q)_\infty\right) \\ \label{eqn_dG2zq} & = \frac{q^2}{1-q^2}(-q^3;q)_\infty.
\end{align}  For the second to  last equality we use the $q$-binomial theorem \eqref{q-binomial} twice, once  with  $z=q$ and once with $z=q^2$.

Finally,   from \eqref{eqn_b3gfdiff}, \eqref{eqn_dG1zq}, and \eqref{eqn_dG2zq} we have that 
\begin{align} 
    \sum_{n\geq 0}b_3(n)q^n
     &=(-q;q)_\infty \sum_{m\geq 2}\frac{q^m}{1+q^m}-\frac{q^2}{1-q^2}(-q^3;q)_\infty.\label{eqn_genG3b3}
\end{align} 
\end{proof}

\subsection{Proof of Theorem \ref{diff} for hooks of length $\mathbf{t=3}$} \label{sec_t3diff} 
To prove Theorem \ref{diff} for hooks of length $t=3$, we find a particular bisection $A(q)+B(q)$ of the generating function  $\sum_{n\geq 0}(a_3(n)-b_3(n))q^n$ established in the prior two subsections.  We then prove separately that $A(q)$ and $B(q)$ have non-negative coefficients of $q^n$ for sufficiently  large  $n$. 
We complete the proof by verifying the theorem directly for the initial coefficients. 

The bisection $A(q)+B(q)$  
is somewhat ad hoc;  on the other hand, it reveals that our bias result for hooks of length $t=3$ ultimately follows from the rather natural partition inequality established in Corollary \ref{rho(n,m) Katriel}.  We welcome  alternative proofs to Theorem \ref{diff} for $t=3$, especially a proof given by a combinatorial injection.

We define the $q$-series
$$A(q) := (-q;q)_\infty \sum_{m=1}^3 \frac{q^{4m}}{1+q^m}-(-q^4;q)_\infty q^4(1+q+2q^3+q^4)$$   and 
\begin{align*}B(q) &:= (-q;q)_\infty \sum_{m=4}^\infty \frac{q^{6m}}{1+q^m} + (-q^9;q)_\infty \left(\frac{q^{41}(1+q^2)(1+q^3)}{1-q^3} + \frac{q^{44}(1+q^3)(1+q^5)}{1-q^5}\right) \\&\hspace{.2in}+(-q^9;q)_\infty(f(q) - g(q)),
\end{align*} 
where the polynomials $f(q)$ and $g(q)$ and their respective coefficients $\{f_j\}$ and $\{g_j\}$  are defined by
\begin{align}\notag  
  f(q) &:= q^{25} +q^{26}+q^{27}+3 q^{28}+5 q^{29}+3 q^{30}+6 q^{31}+7 q^{32}+5 q^{33}+8 q^{34}+7 q^{35}+7 q^{36}\\ \notag &\hspace{.2in}+9 q^{37}+8 q^{38}+7 q^{39} +7 q^{40}+6 q^{41}+6 q^{42}+5 q^{43}+4 q^{44}+3 q^{45}+3 q^{46}+2q^{47}\\ \label{def_aqpoly}  &\hspace{.2in} +q^{48}+q^{49}+q^{50}+q^{51}\\\label{def_aqcoeffs} & =: q^{24}\sum_{j=1}^{27}f_j q^{j},
  \end{align} and 
  \begin{align} 
  g(q)&   :=  q^9 +q^{12}+q^{13}+q^{14}+2 q^{15}+q^{16}+q^{17}+2 q^{18}+q^{19}+2 q^{20}+2 q^{21}+q^{22}   \label{def_bqpoly} \\ \label{def_bqcoeffs} & =:q^8  \sum_{j=1}^{14} g_j q^{j}.\end{align}
\begin{proposition}
    \label{prop_t3gfprime}  With $A(q)$ and $B(q)$ defined above, we have that 
$$\sum_{n\geq 0} (a_3(n) - b_3(n)) q^n = A(q) + B(q).$$
\end{proposition}
Before proving Proposition \ref{prop_t3gfprime}, we establish the non-negativity of the   coefficients of $q^n$ in the series expressions for $A(q)$ and $B(q)$ for sufficiently large values of $n$ by proving Lemma \ref{lem_clm1} and Proposition \ref{lem_abinj} below.

In the remainder of the paper we use the notation $G(q)\succeq_S 0$, where $S\subset  \mathbb N_0$, to mean that when
expanded as a $q$-series, the coefficients of $G(q)$ are non-negative, with the possible exception of the
coefficients of $q^n$ for $n\in S$. When $S = \emptyset$, we simply use the notation $\succeq 0$.
 
\begin{lemma} \label{lem_clm1} We have that 
 \label{clm1} $A(q)  \succeq_{\{5,7\}} 0$.
\end{lemma}

\begin{proposition} \label{lem_abinj} We have that 
 \begin{align}\label{eqn_abinj} (-q^9;q)_\infty (f(q) - g(q)) \succeq_{\{1,2,\dots,75\}} 0. \end{align}
 \end{proposition} 
 \begin{remark} The coefficients of $q^n$ for some $n$ in the exceptional set $\{1,2,\dots,75\}$ appearing in \eqref{eqn_abinj} are in fact non-negative, but the result of Proposition \ref{lem_abinj} is sufficient for our purposes.
 \end{remark}
  We begin with the proof of Proposition \ref{lem_abinj}, which ultimately follows from a special case of our Corollary \ref{rho(n,m) Katriel}.  Following this, we provide an analytic ($q$-series) proof of Lemma \ref{lem_clm1}.
  \medskip \ \\
  {\textit{Proof of Proposition \ref{lem_abinj}.}} Using that $(-q^9;q)_\infty$ is the generating function for $\rho(n,9)$, the statement in Proposition \ref{lem_abinj} is equivalent to the statement that   
\begin{align}\label{eqn_rhoineqt3}  \sum_{j=29}^{42}g_{43-j} \rho(n+j,9) \leq \sum_{j=0}^{26}f_{27-j}\rho(n+j,9) \end{align} 
for $n \geq 25$.  In  Corollary \ref{rho(n,m) Katriel} we take $r=14, s=27, \alpha_k=g_{15-k},$ and $\beta_\ell = f_{28-\ell},$  so that $16 = \sum_{k=1}^{14} \alpha_k < \sum_{\ell=1}^{27} \beta_\ell = 118$.  We further take $\mu_k = k+28, \nu_\ell = \ell-1,$ and $m=9,$  so that by Corollary \ref{rho(n,m) Katriel}, we have that \eqref{eqn_rhoineqt3} holds for $n>N_{\boldsymbol{\alpha},\boldsymbol{\beta},78}$ as in \eqref{eqn:Nmanysubscripts}, which proves the desired result for all but finitely many coefficients.  

Precisely, we have  as in the proof of Theorem \ref{q(n) Katriel} and using Corollary \ref{rho(n,m) Katriel} that $L=\mu_{14}+9(9-1)/2 = 78$ and $\varepsilon = (118-16)/16 = 6.375$.  Experimenting with Mathematica\textsuperscript{TM} \cite{minmaxcode}, we further choose  positive $A,B,C$ satisfying $1/A+1/B+1/C<1$, namely 
$A=180, B=7, C=471177$, so that \eqref{eqn_rhoineqt3} holds for  $n>N_{\boldsymbol{\alpha},\boldsymbol{\beta},78} := N_{180,7,471177}(6.375,78) = 67910.5$, a small enough value to allow us to verify inequality     \eqref{eqn_rhoineqt3} for  $25\leq n \leq 67910$  using SageMath.  This 
 completes the proof of Proposition \ref{lem_abinj}.
  \qed \medskip \ \\
  
\noindent {\textit{Proof of Lemma \ref{lem_clm1}.}}   
We begin by  rewriting $A(q)$  as
\begin{align*}\notag  A(q) \notag  &= 
(-q^2;q)_\infty q^4 + (1+q)q^8(-q^3;q)_\infty  + (1+q)(1+q^2)q^{12}(-q^4;q)_\infty  \\ \notag & \hspace{.2in} - (-q^4;q)_\infty (q^4+q^5+2q^7 + q^8) \\
&= (-q^4;q)_\infty (
 q^6 + 2q^9 + q^{11} + 2 q^{12} +q^{13} + q^{14} + q^{15} - q^5 - q^7).
\end{align*}

We have \begin{align*}(q^5+q^7)(-q^4;q)_\infty& = q^5(-q^5;q)_\infty+q^9(-q^5;q)_\infty+ q^7(-q^5;q)_\infty+q^{11}(-q^5;q)_\infty.
\end{align*}

Since,  \begin{align*}q^9(-q^4;q)_\infty-q^9(-q^5;q)_\infty & = q^{13}(-q^5;q)_\infty\\
q^{11}(-q^4;q)_\infty-q^{11}(-q^5;q)_\infty& =q^{15}(-q^5;q)_\infty,\end{align*} it follows that 
\begin{align*}(-q^4;q)_\infty &(
 q^6 + 2q^9  + q^{11} +  q^{12}  - q^5 - q^7)\\ & = (-q^4;q)_\infty (
 q^6 + q^9  +  q^{12}) +  (-q^5;q)_\infty (q^{13}+q^{15}) -  (-q^5;q)_\infty(q^5 +q^7). \end{align*}
 Next,  $$(q^5 +q^7) (-q^5;q)_\infty =q^5+q^7+q^{10} (-q^6;q)_\infty+q^{12} (-q^6;q)_\infty+q^5\Sigma_6+q^7\Sigma_6,$$
 where we define for $m\in \mathbb N$
 $$\Sigma_m = \Sigma_m(q) := \sum_{k\geq m}q^{k}(-q^{k+1};q)_\infty. $$
 We observe that \begin{align*}q^6(-q^4;q)_\infty - q^{10} (-q^6;q)_\infty& = q^6(-q^5;q)_\infty + q^{15}(-q^6;q)_\infty, \\ 
 q^{12}(-q^4;q)_\infty-q^{12} (-q^6;q)_\infty  &=q^{12}(q^4+q^5+q^9) (-q^6;q)_\infty.
 \end{align*}
 Finally, we have \begin{align*}q^6(-q^5;q)_\infty-q^5\Sigma_6& =q^6+ q^6\Sigma_5-q^5\Sigma_6\\ & = q^6+\sum_{m\geq 5}q^{m+6}(1+q^{m+1})(-q^{m+2};q)_\infty -\sum_{m\geq 6}q^{m+5}(-q^{m+1};q)_\infty\\ & = q^6+ \sum_{m\geq 5}q^{2m+7}(-q^{m+2};q)_\infty
 \end{align*} and similarly 
 \begin{align*}q^9(-q^4;q)_\infty-q^7\Sigma_6& = q^9+ \sum_{m\geq 4}q^{m+9}(q^{m+1}+q^{m+2}+q^{2m+3})(-q^{m+3};q)_\infty.\end{align*}
This completes the proof. 
  \qed 
  \medskip

\subsubsection{Proof of Proposition \ref{prop_t3gfprime}} \label{sec_propt3gfprime}  
We use Propositions \ref{prop_F3gf} and  \ref{prop_G3gf} and find that 
 \begin{align}  \sum_{n\geq 0} &(a_3(n) - b_3(n))q^n \notag \\& =   (-q^3;q)_\infty \frac{q^2+q^3+q^6}{1-q^2} + (-q;q)_\infty\left(\frac{q^6}{1-q^4}+\frac{q^3}{1-q^6}\right)  - (-q;q)_\infty \sum_{m=2}^\infty \frac{q^m}{1+q^m}. \label{eqn_t3gfpr-1} \end{align}
Next we observe that
$$-\sum_{m=1}^\infty \frac{q^m}{1+q^m} = \sum_{m=1}^\infty \frac{q^{4m}}{1+q^m} - \frac{q^3}{1-q^3}-\frac{q}{1-q^2},$$
which follows by a direct calculation, for example, by adding the negative of the series in $m$ on the left-hand-side to the one on the right, using the geometric series formula, and simplifying.
With this, we re-write \eqref{eqn_t3gfpr-1} as
 \begin{align}\label{eqn_Bpart1}   (-q^2;q)_\infty q &+ (-q^3;q)_\infty\frac{q^2+q^3+q^6}{1-q^2}+ (-q;q)_\infty \left(\frac{q^6}{1-q^4} + \frac{q^3}{1-q^6}-\frac{q^3}{1-q^3}-\frac{q}{1-q^2}\right)  \\ \notag &    +(-q;q)_\infty\sum_{m=1}^\infty  \frac{q^{4m}}{1+q^{m}}.     \end{align} 
By straightforward algebra, we find that
\begin{align*}
(1&+q^2)(1+q^3) q  + (1+q^3)\frac{q^2+q^3+q^6}{1-q^2}  \\&\hspace{.2in}+ (1+q)(1+q^2)(1+q^3)\left(\frac{q^6}{1-q^4} + \frac{q^3}{1-q^6}-\frac{q^3}{1-q^3}-\frac{q}{1-q^2}\right).
\\
&=  -\left(q^4(1+q+2q^3+q^4) + \frac{q^9}{1-q^3} \right).
\end{align*}
 We multiply this identity by $(-q^4;q)_\infty$ and insert it into the sum seen in \eqref{eqn_Bpart1}.  
Thus, we have shown that the generating function for $a_3(n) - b_3(n)$ equals 
\begin{align}\label{eqn_a3b3diffv1} A(q) + (-q;q)_\infty \sum_{m=4}^\infty \frac{q^{4m}}{1+q^m} -(-q^4;q)_\infty \frac{q^9}{1-q^3}.\end{align}  Next, we re-write
\begin{align}(-q;q)_\infty \sum_{m=4}^\infty \frac{q^{4m}}{1+q^m} = (-q;q)_\infty \sum_{m=4}^\infty \frac{q^{6m}}{1+q^m}  \label{eqn_A2rewrite} + (-q^3;q)_\infty   \left(q^{16}+q^{17}+q^{18} + \frac{q^{19}}{1-q^5} \right),\end{align}
which may  again be proved  by, for example, subtracting the two infinite series appearing in \eqref{eqn_A2rewrite},  using  geometric series, and simplifying.   We take the series $(-q^3;q)_\infty  q^{19}/(1-q^5)$ appearing in \eqref{eqn_A2rewrite} and subtract from it 
the term $(-q^4;q)_\infty q^9/(1-q^3)$ appearing in \eqref{eqn_a3b3diffv1}
as follows:
  \begin{align}\notag     
  (-q^3;q)_\infty &  \frac{q^{19}}{1-q^5} - (-q^4;q)_\infty \frac{q^9}{1-q^3}  \notag \\ \notag 
  &=   (-q^9;q)_\infty \left((-q^3;q)_6 \frac{q^{19}}{1-q^5} - (-q^4;q)_5 \frac{q^9}{1-q^3}\right) \\ \label{eqn_delp39v1}
  &= (-q^9;q)_\infty \left(
  \frac{q^{41}(1+q^2)(1+q^3)}{1-q^3} + \frac{q^{44}(1+q^3)(1+q^5)}{1-q^5} - p(q)\right),  
  \end{align} 
 where  
 \begin{align*}p(q) :=  q^9 &+ q^{12} +q^{13} + q^{14} + 2q^{15}  +2q^{16} + 2q^{17} + 3q^{18}+2q^{19}+4q^{20}+5q^{21} +4q^{22}+4q^{23}\\&+5q^{24}+5q^{25}+5q^{26}+6q^{27}+5q^{28}+4q^{29}+6q^{30} +4q^{31}+3q^{32} +5q^{33} +2q^{34}\\&+3q^{35}+3q^{36} + q^{38} + q^{39}.\end{align*}
 Inserting \eqref{eqn_A2rewrite}  and then \eqref{eqn_delp39v1} into \eqref{eqn_a3b3diffv1} followed by some straightforward algebra reveals that the generating function for ${a_3(n)-b_3(n)}$ may be written as  
\begin{align*}  
        A(q) &+  (-q)_\infty \sum_{m=4}^\infty \frac{q^{6m}}{1+q^m}  + (-q^3)_\infty   \left(q^{16}+q^{17}+q^{18}   \right)\\&\hspace{.2in} +(-q^9)_\infty \left( \frac{q^{41}(1+q^2)(1+q^3)}{1-q^3} + \frac{q^{44}(1+q^3)(1+q^5)}{1-q^5} - p(q)\right) \\ &= A(q) + B(q)
\end{align*}
as claimed.  In particular, we have also used the easy to verify identity 
$$(-q^3;q)_\infty (q^{16}+q^{17}+q^{18})-(-q^9;q)_\infty p(q) = (-q^9;q)_\infty (f(q)-g(q)).$$
This completes the proof of Proposition \ref{prop_t3gfprime}.\qed \ \\ 
  
\noindent \textit{Proof of Theorem \ref{diff} for $t=3$.} Using Proposition \ref{prop_t3gfprime},  Theorem \ref{diff} for $t=3$ now follows from Lemma \ref{lem_clm1} and Proposition \ref{lem_abinj}, along with a finite computation of the $q$-series coefficients
 in Proposition \ref{prop_t3gfprime} through the $q^{75}$ term, which is easily done (using, for example, Mathematica\textsuperscript{TM}). \qed
\section{Asymptotic formulas}\label{Asymptotic section}

In this section, we prove Theorem \ref{limit}. To this end, we obtain asymptotic formulas for $a_2(n)$, $b_2(n)$, $a_3(n)$, and $b_3(n)$ using the circle method. From these, we immediately obtain asymptotic formulas for $a_2(n)-b_2(n)$ and $a_3(n)-b_3(n)$, allowing us to establish Theorem \ref{limit}. As usual, for two functions $f(n)$ and $g(n)$, we write $f(n) \sim g(n)$ as $n \to \infty$ if  $\frac{f(n)}{g(n)} \to 1$ as $n \to \infty$. Moreover, $\log$ denotes the natural logarithm.

\begin{theorem}\label{2HooksAsymp}
    As $n\to \infty$, we have 
    \begin{align*}
        a_2(n)\sim \frac{3^{5/4}}{8\pi n^{1/4}}e^{\pi\sqrt{\frac n3}}
    \end{align*}
    and
    \begin{align*}
        b_2(n)\sim \frac{3^{1/4}}{4\pi n^{1/4}}e^{\pi\sqrt{\frac{n}{3}}}.
    \end{align*}
\end{theorem}

\begin{theorem}\label{3HooksAsymp}
    As $n\to \infty$, we have 
    \begin{align*}
        a_3(n)\sim \frac{3^{1/4}}{3\pi n^{1/4}} e^{\pi\sqrt{\frac{n}{3}}}
    \end{align*}
    and
    \begin{align*}
        b_3(n)\sim \lp \log(2)-\frac{1}{8}\rp\frac{3^{1/4}}{2\pi n^{1/4}}e^{\pi\sqrt{\frac{n}{3}}}.
    \end{align*}
\end{theorem}

As an immediate consequence of these two theorems, we obtain the following corollaries.

\begin{corollary} \label{2HooksAsymp Corollary}
As $n \to \infty$, we have
\begin{align*}
    a_2(n)-b_2(n)\sim \frac{3^{1/4}}{8\pi n^{1/4}}e^{\pi\sqrt{\frac{n}{3}}}.
\end{align*}
\end{corollary}

\begin{corollary} \label{2HooksRatio}
As $n \to \infty$, we have
\begin{align*}
    a_2(n) \sim \frac32 b_2(n).
\end{align*}
\end{corollary}

\begin{corollary} \label{3HooksAsymp Corollary}
As $n \to \infty$, we have
\begin{align*}
    a_3(n)-b_3(n)\sim \lp\frac{19}{48}-\frac{\log(2)}{2}\rp\frac{3^{1/4}}{\pi n^{1/4}}e^{\pi\sqrt{\frac{n}{3}}}.
\end{align*}
\end{corollary}

\begin{corollary} \label{3HooksRatio}
    As $n \to \infty$, we have
    \begin{align*}
        a_3(n) \sim \frac{2}{3(\log(2)-\tfrac18)}  b_3(n).
    \end{align*}
\end{corollary}

\begin{remark}
    The asymptotic results in Corollaries \ref{2HooksAsymp Corollary} and \ref{3HooksAsymp Corollary} explain the inequalities of Theorem \ref{diff} and prove Theorem \ref{limit}. In addition, Corollaries \ref{2HooksRatio} and \ref{3HooksRatio} prove stronger statements, as will be discussed in Section \ref{Further Conjectures}.
\end{remark}

In light of Theorem \ref{AB}, it is natural to ask about the asymptotics of $a_1(n)$ and $b_1(n)$ as well.  We outline them here since the proofs are easier than those of Theorems \ref{2HooksAsymp} and \ref{3HooksAsymp}. Asymptotics for $b_1(n)$ can be easily derived from \cite[Theorem 1.1]{BiasesDistinct}; in particular, as $n \to \infty$, we have
\begin{align} \label{eqn_b1asy}
    b_1(n) \sim \dfrac{3^{1/4} \log(2)}{2\pi n^{1/4}} e^{\pi \sqrt{\frac{n}{3}}}.
\end{align}
Furthermore, using the generating function identity for $a_1(n) - b_1(n)$ given in \cite[(3.1)]{AndrewsBeck} and {\cite[Proposition 1.8]{NgoRhoades}} (which we restate as Proposition \ref{WrightCircleMethod} below) it can be shown that
\begin{align}\label{eqn_a1mb1asy}
      b_1(n) - a_1(n) \sim \dfrac{3^{1/4} \lp \log(2) - \frac 12 \rp}{2\pi n^{1/4}} e^{\pi \sqrt{\frac{n}{3}}}
\end{align}
as $n \to \infty$.  Using \eqref{eqn_b1asy} and \eqref{eqn_a1mb1asy} along with the definition of $\sim$, we also deduce that 
\begin{align*}
    a_1(n) \sim \dfrac{3^{1/4}}{4\pi n^{1/4}} e^{\pi \sqrt{\frac{n}{3}}}.
\end{align*}
We additionally note, comparing with Corollaries \ref{2HooksRatio} and \ref{3HooksRatio}, that
\begin{align*}
      a_1(n) \sim \dfrac{1}{\log(4)} b_1(n).
\end{align*}

We now discuss Wright's circle method and use it to prove Theorems \ref{2HooksAsymp} and \ref{3HooksAsymp}.

\subsection{Wright's circle method}

The circle method is one of the most important tools in the modern analytic theory of partitions. The method goes back to famous work of Hardy and Ramanujan on $p(n)$ \cite{HardyRamanujan} in which they proved that
\begin{align*}
    p(n) \sim \dfrac{1}{4n\sqrt{3}} e^{\pi \sqrt{\frac{2n}{3}}}
\end{align*}
as $n \to \infty$. Rademacher \cite{RademacherExact} later    extended  their method to prove an exact formula for $p(n)$ similar to, and predating, \eqref{q_Rad}. The presence of non-modular terms in the generating functions of $a_t(n)$ and $b_t(n)$, such as Lambert series and rational functions, means Rademacher's approach cannot be directly used to give exact formulas for $a_t(n)$ or $b_t(n)$. To overcome this difficulty, we use a variation of the circle method due to Wright \cite{Wright} to obtain asymptotic expansions. We use the formulation of this method given by Ngo and Rhoades \cite{NgoRhoades}, who derived asymptotic formulas for the coefficients of products $L(q) \xi(q)$, where $L(q)$ and $\xi(q)$ are certain analytic functions of $q\in\C$ with $|q|<1$ and $q\notin \R_{\le 0}$. Informally, $\xi$ must have a ``main" exponential singularity as $q \to 1$ and $L(q)$ must have polynomial behavior   as $q \to 1$.

To state the formulation of Wright's circle method given in \cite{NgoRhoades}, we require some notation. For $|q|<1$ and $q \not \in \R_{\leq 0}$, we write $q = e^{-z}$ with $z = x + iy$ such that $x > 0$ and $|y| < \pi$. Let $0 < \delta < \frac{\pi}{2}$ and define $D_\delta := \{ z \in \C : \left| \arg \lp z \rp \right| < \frac{\pi}{2} - \delta \}$.  Equivalently, if we let $\Delta = \cot(\delta)$, then $D_\delta := \{ x + iy \in \C : 0 \leq |y| < \Delta x \}$. For fixed $c>0$ and   integer $n \geq 1$, we let $\CC = \CC(c,n)$ be the circle in the complex plane centered at zero with radius $|q| = e^{-x}$, where $x = \frac{c}{\sqrt{n}}$. For  fixed $\delta$, we let $\CC_1 := D_\delta \cap \CC$ and $\CC_2 := \CC \setminus \CC_1$. We call $\CC_1$ the {\it major arc} and $\CC_2$ the {\it minor arc}. We will interchangeably use Landau's big-O and Vinogradov's $\ll$ notations in this section. Recall that for two complex-valued functions $f(z), g(z)$, we write $f = O(g)$ or $f \ll g$ in a region if and only if there is a constant $C>0$ such that $\left| f(z) \right| \leq C \left| g(z) \right|$ for all $z$ in that region.

Let $L(q), \xi(q)$ be analytic functions for $q = e^{-z}$ in the unit disk, and fix $\delta$ and $c$ as above. Following Ngo and Rhoades, we define four hypotheses about the asymptotic properties of these functions. In (H1)--(H4) below, the notation $\ll_\delta$ and $O_\delta$ and indicates that the implied constants depends on $\delta.$

\begin{itemize}
    \item[(H1)] For every positive integer $k$, as $|z| \to 0$ in the cone $D_\delta$ we have
    \begin{align*}
        L(q) = \dfrac{1}{z^B} \lp \sum_{s=0}^{k-1} \alpha_s z^s + O_\delta \lp z^k \rp \rp,
    \end{align*}
    where $B \in \R$ and $\alpha_s \in \C$.  
    \item[(H2)] As $|z| \to 0$ in the cone $D_\delta$ we have
    \begin{align*}
        \xi(q) = K z^\beta e^{\frac{A}{z}}\lp 1 + O_\delta\lp e^{- \frac{\gamma}{z}} \rp \rp,
    \end{align*}
    where $A:=c^2$, and $K, \beta \geq 0$ and $\gamma >0$.\footnote{Ngo and Rhoades assumed that $\gamma > c^2$, but this is not strictly necessary
    when establishing Proposition \ref{WrightCircleMethod}, although the scenario where $\gamma > c^2$ does arise naturally if $\xi$ is modular.}
    \item[(H3)] As $|z| \to 0$ in the region $\frac{\pi}{2} - \delta \leq \left| \arg(z) \right| < \frac{\pi}{2}$, we have
    \begin{align*}
        \left| L(q) \right| \ll_\delta |z|^{-C},
    \end{align*}
 where  $C > 0$.
    \item[(H4)] As $|z| \to 0$ in the region $\frac{\pi}{2} - \delta \leq \left| \arg(z) \right| < \frac{\pi}{2}$, we have
    \begin{align*}
        \left| \xi(q) \right| \ll_\delta \xi\lp |q| \rp e^{- \frac{\delta'}{\mathrm{Re}(z)}},
    \end{align*}
 where $\delta' = \delta'(\delta)>0$.
\end{itemize}

Note that (H1) and (H2) put restrictions of the asymptotic behavior   of $L(q) \xi(q)$ on the major arc; hypotheses (H3) and (H4) put restrictions on the minor arc. Ngo-Rhoades proved the following result from these hypotheses.

\begin{proposition}[{\cite[Proposition 1.8]{NgoRhoades}}] \label{WrightCircleMethod}
     Assume the hypotheses above. Then as $n \to \infty$ we have for any $N \in \Z^+$ that
	\begin{align*}
		c(n) = Ke^{2 \sqrt{A n}} n^{\frac{1}{4}\lp 2B - 2\beta - 3 \rp} \lp \sum\limits_{r=0}^{N-1} p_r n^{-\frac{r}{2}} + O\left(n^{-\frac N2}\right) \rp,
	\end{align*}
	where $p_r := \sum\limits_{j=0}^r \alpha_j c_{j,r-j}$ and $c_{j,r} := \dfrac{(-\frac{1}{4\sqrt{A}})^r \sqrt{A}^{j + \beta-B + \frac 12}}{2\sqrt{\pi}} \dfrac{\Gamma(j + B + \frac 32 + r)}{r! \Gamma(j + B + \frac 32 - r)}$. 
\end{proposition}

Observe that when applying Proposition \ref{WrightCircleMethod}, the precise values of the constants $\delta, \gamma, C$, and $\delta'$ do not appear in the asymptotic formula, and therefore we will not put any emphasis on computing these values exactly in the calculations of this section. In Section \ref{H13}, we verify the hypotheses (H2) and (H4) for the function $\lp -q; q \rp_\infty$, which appears in all the generating functions we consider. In Section \ref{H24 II}, we verify the hypotheses (H1) and (H3) for certain rational functions and Lambert series which appear in the generating functions for $a_2(n)$, $b_2(n)$, $a_3(n)$, and $b_3(n)$.


\subsection{Modular transformations} \label{H13}

Here, we provide the calculation of the asymptotic growth of $\lp -q; q \rp_\infty$ on both the major and minor arcs. As we have seen in previous sections, the generating functions for $a_t(n)$ and $b_t(n)$ naturally arise as the product of $ \lp -q; q \rp_\infty$ and certain other factors. Since $(-q;q)_\infty$ is the only factor with exponential singularities in the generating functions for $a_t(n)$, $b_t(n)$, $t=2,3$,  we set $\xi(q) := \lp -q; q \rp_\infty$.
We also let   $P(q) := \lp q; q \rp_\infty^{-1}$ be the generating function for the partition function $p(n)$. Note here that for $q = e^{2\pi i \tau}$, we have the relation $P\lp e^{2\pi i \tau} \rp = e^{\frac{\pi i \tau}{12}} \eta^{-1}(\tau)$  where $\eta(\tau)$ is the Dedekind $\eta$-function. For the remainder of the section, we use the notation $q = e^{-z}$. The following transformation property for $P(q)$ follows easily from the modular  modular transformation law  $\eta(-1/\tau) = \sqrt{{\tau}/{i}}\ \eta(\tau)$ satisfied by the Dedekind $\eta$-function: 
\begin{align*}
    P(q) = \sqrt{\dfrac{z}{2\pi}} \exp\lp \dfrac{\pi^2}{6z} - \dfrac{z}{24} \rp P\lp e^{- \frac{4\pi^2}{z}} \rp.
\end{align*}
From this formula, it follows that 
\begin{align*}
    \xi(q) = \dfrac{1}{\sqrt{2}} \exp\lp \dfrac{\pi^2}{12z} + \dfrac{z}{24} \rp \dfrac{P\lp e^{- \frac{4\pi^2}{z}} \rp}{P\lp e^{- \frac{2\pi^2}{z}} \rp}.
\end{align*}
Using this transformation law, Jackson and Otgonbayar \cite{BiasesRegular} proved the following asymptotic results on $\xi(q)$ on the major and minor arcs.

\begin{lemma}[{\cite[Lemma 3.8]{BiasesRegular}}] \label{Xi Major Arc}
    Let $\Delta > 0$ be a positive constant, and suppose $z = x + iy$ satisfies $0 \leq |y| < \Delta x$. Then as $z \to 0$ in this region, we have
    \begin{align*}
        \xi(q) = \frac{1}{\sqrt{2}}\exp\lp \dfrac{\pi^2}{12z} + \dfrac{z}{24} \rp \lp 1 + O\lp e^{- \frac{23 \pi^2}{12z}} \rp \rp.
    \end{align*}
\end{lemma}

\begin{lemma}[{\cite[Lemma 3.9]{BiasesRegular}}] \label{Xi Minor Arc}
    Let $\Delta > 0$ be a positive constant, and suppose $z = x + iy$ satisfies $\Delta x \leq |y| < \pi$. Then as $z \to 0$ in this region, we have
    \begin{align*}
        \left| \xi(q) \right| \ll \exp\lp \dfrac{\pi^2}{12x} \lp \dfrac{1}{2} + \dfrac{3}{\pi^2} + \dfrac{6}{\pi^2\lp \Delta^2 + 1 \rp} \rp \rp.
    \end{align*}
    In particular, if $\Delta > 1.45$, then there is a positive constant $\delta'$ such that
    \begin{align*}
        \left| \xi(q) \right| \ll \xi\lp e^{-x} \rp \cdot \exp\lp - \dfrac{\delta'}{x} \rp.
    \end{align*}
\end{lemma}

\begin{remark}
    The results of Jackson and Otgonbayar are stronger than this; in particular, their bounds are completely explicit.
\end{remark}

Notice that as a result of Lemmas \ref{Xi Major Arc} and \ref{Xi Minor Arc}, $\xi(q) = \lp -q; q \rp_\infty$ satisfies (H2) and (H4) with the constants $\beta = 0$, $K=\frac{1}{\sqrt{2}}$ and $A = \frac{\pi^2}{{12}}$.

\subsection{Estimating rational functions and Lambert series} \label{H24 II}

To verify (H1) and (H3) for our situation, we must consider the rational functions and Lambert series that appear in the generating functions for $a_2(n)$, $b_2(n)$, $a_3(n)$, and $b_3(n)$. We define the following functions:
\begin{align*}
L_{o,2}(q)&:=\frac{q^2(1+q+q^3)}{1-q^4},\\
L_{d,2}(q)&:=\frac{q^2}{1-q^2},\\
L_{o,3}(q)&:=\frac{q^3(1+q^3)}{(1+q)(1-q^4)}+\frac{q^6}{1-q^4}+\frac{q^3}{1-q^6},\\
L_{d,3}(q)& := R_{d,3}(q) + \sum_{n \geq 1} \dfrac{q^n}{1 + q^n} := \frac{-q^2}{(1-q^4)(1+q)} - \frac{q}{1+q} + \sum_{n \geq 1} \dfrac{q^n}{1 + q^n}.
\end{align*}
By considering the Laurent expansions of $L_{o,2}(e^{-z})$, $L_{d,2}(e^{-z})$, $L_{o,3}(e^{-z})$, and $R_{d,3}(e^{-z})$, we obtain the following results.

\begin{lemma}\label{H1L2O}
As $|z|\to 0$, we have $$L_{o,2}(q)=\frac{1}{z}\left(\frac{3}{4}+O(z)\right).$$ 
\end{lemma}

\begin{lemma}\label{H1L2D}
As $|z|\to 0$, we have $$L_{d,2}(q)=\frac{1}{z}\left(\frac{1}{2}+O(z)\right).$$
\end{lemma}

\begin{lemma}\label{Hyp1Odd3Hooks}
    As $|z|\to 0$, we have   
    \begin{align*}
         L_{o,3}(e^{-z})=\frac{1}{z}\lp\frac{2}{3}+O(z)\rp.
    \end{align*}
\end{lemma}

\begin{lemma}\label{Rd3H1}
    As $|z|\to 0$, we have $$R_{d,3}(q)=\frac{1}{z}\lp\frac{-1}{8}+O(z)\rp.$$
\end{lemma}

Notice that, as a result of Lemmas \ref{H1L2O} and \ref{H1L2D}, $L_{o,2}$ and $L_{d,2}$ each satisfy (H1) with constants $(B, \alpha_0)=(1,\frac{3}{4})$ and $(1,\frac{1}{2})$, respectively. 
Similarly, by Lemma \ref{Hyp1Odd3Hooks}, $L_{o,3}$ satisfies (H1) with constants $(B,\alpha_0)=(1,\frac{2}{3})$.
Furthermore, by using the triangle and reverse triangle inequalities, we obtain the following bounds on our rational functions, which show that $L_{o,2}$, $L_{d,2}$, and $L_{o,3}$ satisfy (H3).

\begin{lemma}\label{H3L2O}
    Suppose $z=x+iy$. Then, as $|z|\to 0$, we have $$
|L_{o,2}(q)| < \frac{3}{2x} \ll |z|^{-1}.$$
\end{lemma}

\begin{lemma}\label{H3L2D}
    Suppose $z=x+iy$. Then, as $|z|\to 0$, we have $$
|L_{d,2}(q)| < \frac{1}{x} \ll |z|^{-1}.$$
\end{lemma}

\begin{lemma}\label{Hyp3Odd3Hooks}
 Suppose $z=x+iy$. Then, as $|z|\to 0$, we have $$|L_{o,3}(e^{-z})| \ll \dfrac{1}{x} \ll |z|^{-1}.$$
\end{lemma}

\begin{lemma} \label{Rd3H3}
 Suppose $z=x+iy$. Then, as $|z|\to 0$ in the region $\frac{\pi}{2}-\delta\le |y|<\frac{\pi}{2}$, we have $$
|R_{d,3}(q)| < \frac{1}{2x} \ll |z|^{-1}.$$
\end{lemma}

Finally, we estimate the Lambert series appearing in the definition of $L_{d,3}(q)$. Using a technique of Zagier \cite[Section 4]{ZagierMellin} based on Euler--Maclaurin summation, Craig proved the following estimates for this Lambert series.

\begin{lemma}[{\cite[Lemma 4.1]{BiasesDistinct}}] \label{Lambert Major Arc}
    Let $\Delta > 0$ be a constant, and let $z = x + iy$ with $0 \leq |y| < \Delta x$. Then as $z \to 0$ in this region, we have
    \begin{align*}
        \sum_{n \geq 1} \dfrac{e^{-nz}}{1 + e^{-nz}} \sim \dfrac{\log(2)}{z}.
    \end{align*}
\end{lemma}

\begin{lemma}[{\cite[Lemma 4.6]{BiasesDistinct}}] \label{Lambert Minor Arc}
    Let $z = x + iy$ satisfy $x > 0$, $0 \leq |y| < \pi$. Then we have
    \begin{align*}
        \left| \sum_{n \geq 1} \dfrac{e^{-nz}}{1 + e^{-nz}} \right| < \dfrac{1}{x^2}.
    \end{align*}
\end{lemma}

\subsection{Proof of Theorems \ref{2HooksAsymp} and \ref{3HooksAsymp}.}

In this section, we  prove Theorems \ref{2HooksAsymp} and \ref{3HooksAsymp}  by verifying that the hypotheses required for Proposition \ref{WrightCircleMethod} are satisfied by the generating functions  of $a_2(n), b_2(n), a_3(n)$, and $b_3(n)$.

\begin{proof}[Proof of Theorem \ref{2HooksAsymp}] 
  By Lemmas \ref{Xi Major Arc} and \ref{H1L2O}, we have $L_{o,2}(e^{-z})=\frac{3}{4z}+O(1)$ and $\xi(e^{-z})=\frac{1}{\sqrt{2}}\exp\lp\frac{\pi^2}{12z}+O(z)\rp$ on the major arc. These estimates show that $L_{o,2}(q)$ and $\xi(q)$ satisfy Hypotheses (H1) and (H2) with constants $\beta=0$, $K=\frac{1}{\sqrt{2}}$, $A=\frac{\pi^2}{12}$, $\alpha_0=\frac{3}{4}$, and $B=1$. Furthermore, on the minor arcs, Lemmas \ref{Xi Minor Arc} and \ref{H3L2O} show that our generating function satisfies Hypotheses (H3) and (H4).

  Similarly, Lemmas \ref{Xi Major Arc}, \ref{H1L2D}, \ref{Xi Minor Arc}, and \ref{H3L2D} show that $L_{d,2}(q)$ and $\xi(q)$ satisfy Hypotheses (H1)-(H4) with constants $\beta=0$, $K=\frac{1}{\sqrt{2}}$, $A=\frac{\pi^2}{12}$, $\alpha_0=\frac{1}{2}$, and $B=1$. 

Therefore, we can apply  Proposition \ref{WrightCircleMethod} to obtain the desired asymptotics for $b_2(n)$ and $a_2(n)$.
\end{proof}
Recall that, by Proposition \ref{eqn_genG3b3}, 
\begin{align*}
    \sum_{n\ge 0} b_3(n)q^n&=(-q;q)_\infty\sum_{m\ge 2} \frac{q^m}{1+q^m}-\frac{q^2}{1-q^2}(-q^3;q)_\infty\\
    &=\xi(q)\left(\frac{-q}{1+q}+\frac{-q^2}{(1-q^2)(1+q)(1+q^2)}+\sum_{m\ge 1} \frac{q^m}{1+q^m}\right)\\
    &{=\xi(q)L_{d,3}(q)}
\end{align*}

\begin{proof}[Proof of Theorem \ref{3HooksAsymp}]
Lemmas \ref{Xi Major Arc}, \ref{Hyp1Odd3Hooks}, \ref{Xi Minor Arc}, and \ref{Hyp3Odd3Hooks} show that $L_{o,3}(q)$ and $\xi(q)$ satisfy Hypotheses (H1)-(H4) with constants $\beta=0$, $K=\frac{1}{\sqrt{2}}$, $A=\frac{\pi^2}{12}$, $\alpha_0=\frac{2}{3}$, $B=1$.

Furthermore, by Lemmas \ref{Xi Major Arc}, \ref{Lambert Major Arc}, and \ref{Rd3H1}, we have $$L_{d,3}(e^{-z})=\frac{1}{z}\lp \log(2)-\frac{1}{8}\rp+O(1)$$ and $\xi(e^{-z})=\frac{1}{\sqrt{2}}\exp\lp\frac{\pi^2}{12z}+O(z)\rp$ on the major arc. These estimates show that $L_{d,3}(q)$ and $\xi(q)$ satisfy Hypotheses (H1) and (H2) with constants $\beta=0$, $K=\frac{1}{\sqrt{2}}$, $A=\frac{\pi^2}{12}$, $\alpha_0=\log(2)-\frac{1}{8}$, and $B=1$. Furthermore, on the minor arcs, Lemmas \ref{Xi Minor Arc}, \ref{Lambert Minor Arc} and \ref{Rd3H3} show that our generating function satisfies Hypotheses (H3) and (H4).

Therefore, we can apply  Proposition \ref{WrightCircleMethod} to obtain the desired asymptotics for $b_3(n)$ and $a_3(n)$. 
\end{proof}

\section{Further bias results} \label{Further bias}

Recall that we denote by  $\ell_1(\lambda)$ (respectively $\ell_2(\lambda)$)   the number of parts $\lambda_i$ of $\lambda$ with $\lambda_i-\lambda_{i+1}=1$ (respectively $\lambda_i-\lambda_{i+1}=2$).  We assume $\lambda_k=0$ if $k>\ell(\lambda)$. For $j=1,2$, we refer to $\ell_j(\lambda)$,  as the number of gaps of size $j$ in $\lambda$.
 Note that the $\ell_2(\lambda)$ partition statistic appeared in our calculations of both $a_3(n)$ and $b_3(n)$. Thus, it is natural to investigate a possible bias in the total number of gaps of size exactly $1$, respectively $2$, in odd versus distinct partitions. We prove that such a bias exists. As in the case with the total number hooks of fixed  length, the direction of the bias for the  total number of gaps   of size $2$ is the opposite of that  for the total number of gaps of size $1$.

 \begin{theorem} \label{beck_l1} For $n\in \mathbb N_0$, $$\sum_{\lambda\in \mathcal D(n)}\ell_1(\lambda) - \sum_{\lambda\in \mathcal O(n)}\ell_1(\lambda)$$ is non-negative  except for $n=2$ and $n=4$ in which case it equals $-1$.
\end{theorem} 
\begin{theorem} \label{Beck} For $n\in\mathbb N_0$,   $$ \displaystyle \sum_{\lambda\in \mathcal O(n)}\ell_2(\lambda)-\sum_{\lambda\in \mathcal D(n)}\ell_2(\lambda) $$
is non-negative except for $n=2$ and $n=6$ in which case it equals $-1$.
\end{theorem}

After proving each theorem, we provide combinatorial interpretations of the respective excesses.

\subsection{Proof of Theorem~\ref{beck_l1}} 
  If $\lambda\in \mathcal O(n)$, then $$\ell_1(\lambda)= \begin{cases}1 & \text{ if } 1\in \lambda,\\ 0 & \text{ else.} \end{cases}$$ Thus,  $$\sum_{n\geq 0}\sum_{\lambda\in \mathcal O(n)}\ell_1(\lambda) q^ n = q\frac{1}{(q;q^2)_\infty}=q(-q;q)_\infty.$$

The number of gaps of size $1$ all in distinct partitions of $n$ is equal to the total number of parts in all distinct partitions of $n$ minus the total number of gaps of size at least $2$ in all distinct partitions of $n$, i.e., $b_1(n)-b_2(n)$.   Using this and the generating function for $b_2(n)$ found in \eqref{eqn_b2}, we have  $$\sum_{n\geq 0}\sum_{\lambda\in \mathcal D(n)}\ell_1(\lambda) q^n = (-q;q)_\infty \sum_{m\geq 1}\frac{q^m}{1+q^m}-\frac{q^2}{1-q}(-q^2;q)_\infty.$$ 

Therefore,
\begin{align}\notag
    \sum_{n\ge 0}&\left(\sum_{\lambda\in \mathcal D(n)}\ell_1(\lambda) - \sum_{\lambda\in \mathcal O(n)}\ell_1(\lambda)\right)q^n\\&=\notag (-q;q)_\infty \sum_{m\geq 1}\frac{q^m}{1+q^m}  -\frac{q^2}{1-q}(-q^2;q)_\infty- q(-q;q)_\infty\\  &= (-q;q)_\infty\left(\sum_{m\geq 1}\frac{q^m}{1+q^m}-\frac{q^2}{1-q^2}-q \right)\label{eqn_ell1r1}.
\end{align}

We will show that the expression in \eqref{eqn_ell1r1}, when expanded as a $q$-series, has non-negative coefficients for $n\geq 5$.  We  rewrite \eqref{eqn_ell1r1} as\begin{align} \label{eqn_ell1rewrite}
 (-q;q)_\infty\left(\sum_{m\geq 1}\frac{q^{3m}}{1+q^m}-\frac{q^2}{1+q} \right),
\end{align} 
which can be seen by, for example, subtracting \eqref{eqn_ell1rewrite} from \eqref{eqn_ell1r1} and using the geometric series identity.

Next we write \begin{align}\notag
 (-q; q)_\infty &\left(\sum_{m\geq 1}\frac{q^{3m}}{1+q^m} -\frac{q^2}{1+q} \right)\\ \notag & =(-q;q)_\infty\left(\frac{q^3}{1+q}+\frac{q^6}{1+q^2}-\frac{q^2}{1+q} \right) + (-q;q)_\infty\sum_{m\geq 3}\frac{q^{3m}}{1+q^m}\\  & = (-q^3;q)_\infty\left(q^3+q^5+q^6+q^7-q^2-q^4\right) + (-q;q)_\infty\sum_{m\geq 3}\frac{q^{3m}}{1+q^m}. \label{triple} \end{align}

To finish the proof, we show that the only negative terms in $$(-q^3;q)_\infty \left(q^3+q^5+q^6+q^7-q^2-q^4\right),$$ when expanded as a $q$-series, are $-q^2$ and $-q^4$.

 By expanding the $q$-Pochhammer symbol, we obtain \begin{align*}q^2(-q^3;q)_\infty& =q^2+q^5(-q^4;q)_\infty +q^2\sum_{k\geq4}q^k(-q^{k+1};q)_\infty\\  & = q^2+q^5(-q^4;q)_\infty + q^3\sum_{k\geq 3}q^{k}(-q^{k+2};q)_\infty.\end{align*} 

Similarly,
\begin{align*}q^3(-q^3;q)_\infty &=q^3+q^3\sum_{k\geq 3}q^kq^{k+1}(-q^{k+2};q)_\infty+ q^3\sum_{k\geq 3}q^k(-q^{k+2};q)_\infty,\\ q^4 (-q^3;q)_\infty &=  q^4+q^7(-q^5;q)_\infty+q^8(-q^5;q)_\infty +q^{11}(-q^5;q)_\infty+q^6\sum_{k\geq 3}q^k(-q^{k+3};q)_\infty,\\ q^5(-q^3;q)_\infty&= q^5(-q^4;q)_\infty + q^8(-q^5;q)_\infty+ q^{12}(-q^5;q)_\infty, \\ q^6(-q^3;q)_\infty &= q^6+ q^6\sum_{k\geq 3}q^{2k+1}(1+q+q^{k+2})(-q^{k+3};q)_\infty+ q^6\sum_{k\geq 3}q^k(-q^{k+3};q)_\infty, \\ q^7(-q^3;q)_\infty &= q^7(-q^5;q)_\infty + q^{10}(-q^4;q)_\infty + q^{11}(-q^5;q)_\infty.
\end{align*}

 Then, \begin{align*} (-q^3;q)_\infty & \left(q^3+q^5+q^6+q^7-q^2-q^4\right)  =  -q^2+q^3-q^4+H_1(q),\end{align*} where \begin{align*} H_1(q) :=& \  q^6+q^{10}(-q^4;q)_\infty+q^{12}(-q^5;q)_\infty\\ & +q^3\sum_{k\geq 3}q^kq^{k+1}(-q^{k+2};q)_\infty+ q^6\sum_{k\geq 3}q^{2k+1}(1+q+q^{k+2})(-q^{k+3};q)_\infty.\end{align*} 

Clearly, $H_1(q)$  expanded as a $q$-series has non-negative coefficients. 
This completes the proof of Theorem \ref{beck_l1}.  \qed \medskip

 The proof of Theorem \ref{beck_l1} shows that 
\begin{align*}\sum_{n\geq 0}\left(\sum_{\lambda\in \mathcal D(n)}\ell_1(\lambda) - \sum_{\lambda\in \mathcal O(n)}\ell_1(\lambda)\right)q^n  =   -q^2+q^3-q^4+H_1(q) + (-q;q)_\infty \sum_{m\geq 3} \frac{q^{3m}}{1+q^m}.  
    \end{align*}
We now give  a combinatorial interpretation for $\displaystyle\sum_{\lambda\in \mathcal D(n)}\ell_1(\lambda) - \sum_{\lambda\in \mathcal O(n)}\ell_1(\lambda)$  when $n\geq 5$. 

First note that $$(-q;q)_\infty \sum_{m\geq 3} \frac{q^{3m}}{1+q^m}$$ is the generating function for the number of partitions of $n$ in which exactly one part greater than $2$ has multiplicity three and all other parts have multiplicity one.  Next, we interpret the summands in $H_1(q)$ as combinatorial generating functions:

\begin{itemize} 
\item $q^6$ is the generating function for the number of partitions of $n$ with exactly three parts equal to $2$ (and no other parts).\smallskip

\item $q^{10}(-q^4;q)_\infty$ is the generating function for the number of partitions $\lambda$ of $n$ with exactly three parts equal to $2$, all other parts distinct, and $1,3 \in \lambda$. \smallskip

\item $q^{12}(-q^5;q)_\infty$ is the generating function for the number of partitions $\lambda$ of $n$ with exactly three parts equal to $1$, all other parts distinct, and $2,3,4 \in \lambda$. \smallskip

\item $q^3\sum_{k\geq 3}q^kq^{k+1}(-q^{k+2};q)_\infty$ is the generating function for the number of partitions $\lambda$ of $n$ with $\ell(\lambda)\geq 5$, exactly three parts equal to $1$, all other parts distinct, $2\not \in \lambda$, and the two smallest parts not equal to $1$ differing by one. \smallskip

\item $q^6\sum_{k\geq 3}q^{2k+1}(1+q+q^{k+2})(-q^{k+3};q)_\infty$ equals $q^6\sum_{k\geq 3}q^kq^{k+1}(-q^{k+2};q)_\infty+ q^6\sum_{k\geq 3}q^kq^{k+2}(-q^{k+3};q)_\infty,$ which  is the generating function for the number of partitions $\lambda$ of $n$ with $\ell(\lambda)\geq 5$, exactly three parts equal to $2$, all other parts distinct, $1\not \in \lambda$, and the two smallest parts not equal to $2$ differing by one or two.  
\end{itemize}

Thus, for $n \geq 5$, $$\sum_{\lambda\in \mathcal D(n)}\ell_1(\lambda) - \sum_{\lambda\in \mathcal O(n)}\ell_1(\lambda)$$ equals the number of partitions $\lambda$ of $n$ with exactly one part repeated three times and all other parts distinct such that 
\begin{itemize}
\item[(i)] if the repeated part is $1$, then $\ell(\lambda)\geq 5$, the two smallest parts not equal to $1$ differ by one, and if $2\in \lambda$, then $4\in \lambda$. 

\item[(ii)] if the repeated part is $2$, then $\ell(\lambda)=3$ or $\ell(\lambda)\geq 5$, and  
\begin{itemize}
\item[$\bullet$] if $1\in \lambda$, then $3 \in \lambda$ 

\item[$\bullet$] if $1\not \in \lambda$, then the two smallest parts not equal to
$2$ differ by one or two. 
\end{itemize}
\end{itemize}

\begin{corollary}\label{Cor-l1} For $n\geq 5$, $$\sum_{\lambda\in \mathcal D(n)}\ell_1(\lambda) - \sum_{\lambda\in \mathcal O(n)}\ell_1(\lambda)\leq b_1(n)-a_1(n).$$
\end{corollary}\begin{proof}This follows immediately from the combinatorial interpretation above and Corollary \ref{cor_andt1}. \end{proof}

 \begin{remark}Corollary \ref{Cor-l1} can be rewritten as \begin{align}\label{cor_64} a_1(n) - \sum_{\lambda\in \mathcal O(n)}\ell_1(\lambda)\leq b_1(n)-\sum_{\lambda\in \mathcal D(n)}\ell_1(\lambda), \ \ \text{ for } n\geq 5.\end{align} In fact,   the inequality is true for all $n\geq 0$.  Combinatorially, \eqref{cor_64}  states that  the total number of different part sizes greater than $1$ in all odd partitions of $n$ is at most  the total number of gaps greater than $1$ in all distinct partitions of $n$.   It is interesting to compare this with  Corollary \ref{cor_andt1} which implies that the total number of parts in all distinct partitions of $n$ is at most the total number of different part sizes in all odd partitions of $n$. \end{remark}
\subsection{Proof of Theorem~\ref{Beck}}

From Section \ref{sec_t=3},  $$\sum_{\lambda\in \mathcal D(n)}\ell_2(\lambda)= (-q^3;q)_\infty \frac{q^2}{1-q^2} \text{\ \  and } \sum_{\lambda\in \mathcal O(n)}\ell_2(\lambda)=(-q;q)_\infty    \frac{q^4}{1-q^4}.$$
Thus, we need to show that,   the coefficients  of $q^n$, $n\neq 2,6$, in 
\begin{align*} H_2(q):=\sum_{n\geq 0} \left(\sum_{\lambda\in \mathcal O(n)}\ell_2(\lambda)-\sum_{\lambda\in \mathcal D(n)}\ell_2(\lambda)\right) q^n=(-q;q)_\infty    \frac{q^4}{1-q^4} - (-q^3;q)_\infty \frac{q^2}{1-q^2}\end{align*}  are non-negative.  A direct calculation shows that the coefficients of $q^2$ and $q^6$ in $H_2(q)$ are both equal to $-1$.

After some  $q$-series manipulations, we have that
\begin{align*}\label{b2}H_2(q)+q^2+q^6& =(-q^3;q)_\infty \frac{q^5}{1-q^2} -q^5(-q^4;q)_\infty - q^2(1+q^4)\sum_{m\geq 5}q^m(-q^{m+1};q)_\infty\\ & = (-q^4;q)_\infty \frac{q^7}{1-q} - q^2(1+q^4)\sum_{m\geq 5}q^m(-q^{m+1};q)_\infty\\ & = (-q^4;q)_\infty \sum_{k\geq 7} q^k-(1+q^4)\sum_{k\geq 7}q^k(-q^{k-1};q)_\infty\\ & =(1+q^4)\sum_{k\geq 7}q^k\left((-q^5;q)_\infty-(-q^{k-1};q)_\infty\right).  \end{align*}

Moreover, for $k\geq 7$, we have 
\begin{align*}
    (-q^5;q)_\infty - (-q^{k-1};q)_\infty &=  \sum_{j\geq 5}q^j (-q^{j+1};q)_\infty  -  \sum_{j\geq k-1}q^j (-q^{j+1};q)_\infty \\ &  
    = \sum_{j=5}^{k-2}q^j(-q^{j+1};q)_\infty.
\end{align*}
Thus,  \begin{equation}\label{H} H_2(q)+q^2+q^6=(1+q^4)\sum_{k\geq 7}q^k \sum_{j=5}^{k-2}q^j(-q^{j+1};q)_\infty,\end{equation} which clearly has non-negative coefficients and completes the proof. \qed \medskip
 
 We interpret  \eqref{H} as a combinatorial generating function as follows. 
 We  write $k\geq 7$ as $k=3d+r$, $0\leq r\leq 2$ and we view $q^k$ as generating $d$ parts equal to $3$ and one part equal to $r$. The sum $\sum_{j=5}^{k-2}q^j(-q^{j+1};q)_\infty$ generates non-empty distinct partitions with smallest part between $5$ and $k-2$ (inclusive).

Thus, for $n\neq 2, 6$, $$\sum_{\lambda\in \mathcal O(n)}\ell_2(\lambda)-\sum_{\lambda\in \mathcal D(n)}\ell_2(\lambda)$$
 equals the number of partitions $\lambda$ of $n$ satisfying all of the following conditions: 
\begin{enumerate}
\item[(i)] $m_\lambda(3)\geq 2$, and if $m_\lambda (3)= 2$, then $\lambda_{\ell(\lambda)}<3$,

\item[(ii)]   $3$ is the only repeated part,  and 
  $1$ and $2$ cannot both occur as parts in $\lambda$,

\item[(iii)] $\lambda$ has part greater than $4$ and the smallest  part $s$ greater than $4$ satisfies $5\leq s\leq \left(\sum_{\lambda_i\leq 3} \lambda_i\right)-2$.
\end{enumerate}

We also provide a second combinatorial interpretation for the excess of Theorem \ref{Beck}.  From \eqref{H}, we have  that 
\begin{align*}H_2(q)+q^2+q^6& = (1+q^4)\sum_{k\geq 7}q^k\sum_{j=6}^{k-1} q^{j-1}(-q^j;q)_\infty \\
&= (1+q^4)\sum_{k= 6}^\infty q^{k-1}(-q^k;q)_\infty \sum_{j=k+1}^\infty q^j \\
&=(1+q^4)\sum_{k=6}^\infty \frac{q^{2k}}{1-q}(-q^k;q)_\infty\\  & = \frac{q^2}{1-q}(1+q^4)\sum_{k=5}^\infty q^{2k}(-q^{k+1};q)_\infty.\end{align*} Thus, for $n\neq 2, 6$, $$\sum_{\lambda\in \mathcal O(n)}\ell_2(\lambda)-\sum_{\lambda\in \mathcal D(n)}\ell_2(\lambda)$$
also equals the number of partitions $\lambda$ of $n$ such that  $1$ occurs at least twice as a part of $\lambda$, there are no parts equal to $2$ or $3$, the smallest part greater than $4$ occurs twice, and all other parts occur once.
 
We end this section by noting that the combinatorial interpretations for the excess in Theorem \ref{beck_l1}, respectively \ref{Beck} are of a similar flavor to  Theorem \ref{AB}.

\section{Further conjectures} \label{Further Conjectures}

We conclude this paper with several remarks and conjectures. In particular, we explore possible generalizations of our main results to other values of $t$ and other families of partitions. We also remark briefly on an apparent congruence which appears in our data.

\subsection{A stronger conjecture}
Recall Conjecture \ref{Odd-Distinct Conjecture} (ii), which says that for all integers $t \geq 2$, 
$a_t(n) - b_t(n) \to \infty$ as $n \to \infty$.  Corollaries \ref{2HooksAsymp Corollary}
and \ref{3HooksAsymp Corollary} establish Conjecture \ref{Odd-Distinct Conjecture} (ii) for $t=2$ and $t=3.$
We also conjecture the following stronger asymptotic result, which is proved for $t=2$ and $t=3$ in Corollaries \ref{2HooksRatio} and \ref{3HooksRatio}.  
\begin{conjecture} \label{Strong Odd-Distinct Conjecture}
For every integer $t \geq 2$, there is a positive constant $\alpha_t > 1$ such that $a_t(n) \sim \alpha_t b_t(n)$ as $n \to \infty$.
\end{conjecture}

We now give a heuristic explanation for why one should think that Conjecture \ref{Strong Odd-Distinct Conjecture} is true. Due to Euler, we know that the number of partitions into odd parts equals the number of partitions into distinct parts. Thus, the total number of cells among the two families of partitions are also equal and, as noted in Section \ref{intro}, we have for every $n$ that 
\begin{align} \label{Conjecture Eqn}
    \sum_{t \geq 1} a_t(n) = \sum_{t \geq 1} b_t(n).
\end{align}
Corollary \ref{cor_andt1} states that $b_1(n) > a_1(n)$, and the main theorems of our paper establish that $a_t(n) > b_t(n)$ for $t=2,3$ and $n$ sufficiently large. To explain this phenomenon informally, let $\lambda \in \mathcal O(n)$ and $\mu \in \mathcal D(n)$ be random partitions of the same (sufficiently large) integer $n$. Theorem \ref{AB} implies that, on average, $\mu$ should have more hooks of length 1 than $\lambda$. In order to balance this fact with \eqref{Conjecture Eqn}, $\lambda$ must have more hooks of other lengths $t$. Since there are in general more ``small" hook lengths than ``large" hook lengths in a given partition, we should expect the ``small" hook lengths to be the main contribution to balancing \eqref{Conjecture Eqn}. Our results do not consider what happens if $t$ is allowed to vary with $n$, so we do not know whether a random odd partition or a random distinct partition should have more ``large" hook numbers. However, the data we obtain, in particular the fact that $N_t$ appears to grow with $t$, suggests that for $n \leq N_t$ it is usually true that $b_t(n) > a_t(n)$, contrary to the case for $n \gg 0$. This suggests that the hook numbers of partitions into distinct parts tend to be either equal to $1$ or  large relative to $n$, while partitions into odd parts tend to have hook numbers of more intermediate values. It would be interesting to understand this phenomenon in more detail. For example, given a positive constant $0 < \varepsilon < 1$, one could ask whether odd or distinct partitions have more hook numbers of size $t \geq n^\varepsilon$. Our heuristic suggests that for large enough $\varepsilon$ in this interval, hooks of length $t \geq n^\varepsilon$ should be more common in distinct partitions of $n$ than in odd partitions of $n$. It would also be interesting to understand this picture in terms of probability distribution functions. Such a possibility is discussed further in Section \ref{Future directions}.

\subsection{Hook bias conjectures for self-conjugate versus distinct odd parts partitions}

Since the number of self-conjugate partitions of $n$ equals the number of partitions of $n$ into distinct odd parts, one may study the possible bias in the number of hooks of fixed length in these sets of partitions. To this end, let $a_t^*(n)$ be the total number of hooks of length  $t$ in all self-conjugate partitions of $n$ and let $b_t^*(n)$ be the total number of hooks of length $t$ in all partitions of $n$ into distinct odd parts. We propose the following conjecture, which is the analog of Conjecture \ref{Odd-Distinct Conjecture}.

\begin{conjecture} \label{conj-sc-do}
For every integer $t \geq 2$, there is an integer $N_t^*$ such that for all $n > N_t$, we have $a^*_t(n) \geq b_t^*(n)$. Furthermore, we have $a_t^*(n) - b_t^*(n) \to \infty$ as $n \to \infty$.
\end{conjecture}

The following table gives conjectural values of the constants $N_t^*$ for $2 \leq t \leq 10$. \\

\begin{center}
\begin{tabular}{|c|c|c|c|c|c|c|c|c|c|} \hline
$t$ & 2 & 3 & 4 & 5 & 6 & 7 & 8 & 9 & 10 \\ \hline
$N_t^*$ & 10 & 8 & 22 & 12 & 30 & 20 & 38 & 32 & 54 \\ \hline
\end{tabular}

\vspace{0.05in}
{\small Table: Conjectural values for $N_t^*$.}
\end{center}

The explanation of this conjecture is, on heuristic grounds, the same as that for Conjecture \ref{Odd-Distinct Conjecture}. In particular, it appears that the hook numbers of partitions into distinct odd parts tend to be either very small or very large, whereas the hook numbers of self-conjugate partitions tend to take intermediate values. It would be equally natural to investigate such questions for other well-known partition identities. For example, one could consider the Rogers-Ramanujan identities or Glaisher's identity involving $k$-regular partitions and partitions with parts repeated at most $(k-1)$ times.

\subsection{Congruence conjectures}

We now consider $a_t^*(n)$, the number of hooks of length $t$ in all self-conjugate partitions of $n$, from an arithmetic point of view. By the symmetry of self-conjugate partitions, it is clear that $a_{2m}^*(n) \equiv 0 \pmod{2}$ for all $n$; this is because hook numbers on the main diagonal of a self-conjugate partition are necessarily (distinct) odd integers. While generating data in support of Conjecture \ref{conj-sc-do}, the authors discovered what appear to be nontrivial congruence relations for $a^*_t(n)$ extending the trivial observation above.

\begin{conjecture}
We have for all $n \geq 0$ and $m \geq 1$ that
\begin{align*}
   a_{2m}^*(n) \equiv 0 \pmod{2m}.
\end{align*}
\end{conjecture}

This has been verified by computer for $1 \leq m \leq 5$ and $0 \leq n \leq 70$. We note that if we denote by $a_t^{**}(n)$  the number of hooks of length divisible by $t$  in all self-conjugate partitions of $n$, then by using the sieve of Eratosthenes one may show that the conjecture above is equivalent to showing that $a_{2m}^{**}(n) \equiv 0 \pmod{2m}$ for all $m \geq 1$ and $n \geq 0$.

\subsection{Future directions} \label{Future directions}
There are a number of statistical results in the literature regarding how hook numbers distribute among unrestricted partitions. For example, see \cite{AyyerSinha, BringmannCraigMalesOno, GriffinOnoTsai, LangWanXu} for several different approaches to such distributions. Adopting the same philosophy as this paper, it would be natural to ask  how these statistics behave in families of partitions such as partitions into odd parts or partitions into distinct parts. For example, based on our work and \cite{AyyerSinha} it would be natural to ask whether the 2-cores and 3-cores of partitions into distinct parts might be larger on average than those of partitions into odd parts. It would also be natural to ask similar questions for partition statistics other than those statistics discussed in this paper; a few potentially interesting examples might be the largest/smallest part or rank and crank statistics.

\subsection*{Data availability}
Data sharing not applicable to this article as no datasets were generated or analysed during the current study.

\bibliographystyle{plain}
\bibliography{paper}
\ \\
\end{document}